\documentclass[final,leqno]{siamart171218}
\usepackage{amsmath,amsfonts,amssymb}
\usepackage{mathtools}
\usepackage{float}
\usepackage{geometry}
\usepackage{graphicx}
\usepackage{hyperref,url}
\usepackage{tikz,xypic,subcaption}
\newtheorem{thm}{Theorem}

\newtheorem{lemm}{Lemma}[thm]
\newtheorem{prop}{Proposition}
\newtheorem{defn}{Definition}
\newtheorem{rem}{Remark}

\newcommand{\norm}[1]{\left\lVert#1\right\rVert}

\begin{document}

\title{
Stability of a Nonlocal Traffic Flow Model\\ for Connected Vehicles\thanks{Submitted to the editors
DATE. \funding{This 
work is supported in part by the NSF DMS-2012562,  DMS-1937254
and  ARO MURI Grant W911NF-15-1-0562.
}}}

\headers{Stability of a Nonlocal Traffic Flow Model}{Kuang Huang and Qiang Du}

\author{Kuang Huang\thanks{Department of Applied Physics and Applied Mathematics, 
Columbia University, New York, NY 10027; {\tt kh2862@columbia.edu}} \and Qiang Du\thanks{Department of Applied Physics and Applied Mathematics,  and Data Science Institute,
Columbia University, New York, NY 10027; {\tt qd2125@columbia.edu}}}

\maketitle

\begin{center}
{\small \color{purple} This work has been published in \emph{SIAM
Journal on Applied Mathematics}, 82(1), 221-243. Please refer to the official publication for citation.}
\end{center}

\begin{abstract}
	The emerging connected and automated vehicle technologies allow vehicles to perceive and process traffic information in a wide spatial range. Modeling nonlocal interactions between connected vehicles and analyzing their impact on traffic flows become important research questions to traffic planners. This paper considers a particular nonlocal LWR model that has been studied in the literature. The model assumes that vehicle velocities are controlled by the traffic density distribution in a nonlocal spatial neighborhood. By conducting stability analysis of the model, we obtain that, under suitable assumptions on how the nonlocal information is utilized, the nonlocal traffic flow is stable around the uniform equilibrium flow and all traffic waves dissipate exponentially. Meanwhile, improper use of the nonlocal information in the vehicle velocity selection could result in persistent traffic waves. Such results
	can shed light to the future design of driving algorithms for connected and automated vehicles.
\end{abstract}



\section{Introduction}\label{sec:intro}
\setcounter{equation}{0}

In transportation research, one of the central problems is to understand the collective behavior of moving vehicles. 
With the development of connected vehicle technology, vehicles on the road, such as those traveling on the highway, can be connected through some vehicle-to-vehicle (V2V) or vehicle-to-infrastructure (V2I) communication networks \cite{dey2016vehicle}. As a result, each vehicle perceives nonlocal information on the road. The enhanced access to traffic information brings new opportunities and challenges on many aspects of traffic flows, ranging from traffic management, communication infrastructure and protocols to vehicle design and control. Theoretical studies and modeling efforts are also in great need \cite{huang2020scalable}. On one hand, new models are imperative to study how nonlocal information affects traffic flows and to explore the emergence of new traffic phenomena; On the other hand, car manufacturers will face the problem of designing driving algorithms to guide  connected vehicles. This is an interactive and iterative process: a good algorithm is expected to utilize nonlocal information to improve the traffic, and  at the same time, a good model can guide the algorithm design. 

On the macroscopic level, traffic flows on highways have been modeled via continuum descriptions and hyperbolic conservation laws \cite{lighthill1955kinematic,richards1956shock,payne1971model,aw2000resurrection} similar to models of continuum media. Our present study focuses on such continuum descriptions of the dynamics of vehicle densities on a ring road. The main aim of this work is a mathematical demonstration of how  nonlocal information can be utilized to gain desired benefits. By conducting stability analysis of a nonlocal macroscopic traffic flow model, we offer evidence of traffic wave stabilization with proper usage of nonlocal vehicle density information in velocity control.

\subsection{Background mathematical models}\label{sec:background}
 
A common block of macroscopic traffic flow models is the continuity equation:
\begin{align}
	\partial_t\rho(x,t)+\partial_x\left(\rho(x,t)u(x,t)\right)=0,\label{eq:continuity}
\end{align}
that describes the conservation of vehicles, where $\rho(x,t)$ and $u(x,t)$ denote the aggregated traffic density and velocity and $x$ and $t$ are spatial and temporal coordinates.  The Lighthill-Whitham-Richards (LWR) model \cite{lighthill1955kinematic,richards1956shock} is the most extensively used macroscopic traffic flow model. It assumes a fundamental relation:
\begin{align}
 	u(x,t)=U\left(\rho(x,t)\right),\label{eq:relation}
\end{align}
between traffic density and velocity, meaning that the driving speed of a vehicle is determined only by the instantaneous density at the vehicle's current
location. The function $U(\cdot)$ is also referred to as the \emph{desired speed function}.
The LWR model follows from \eqref{eq:continuity}\eqref{eq:relation} as a scalar conservation law:
\begin{align}
	\partial_t\rho(x,t)+\partial_x\left(\rho(x,t) U\left(\rho(x,t)\right)\right)=0.\label{eq:lwr}
\end{align}
The LWR model \eqref{eq:lwr} may produce shock wave solutions even with smooth initial data. Such shock wave solutions qualitatively explain the formation and propagation of traffic jams.

\subsection{Nonlocal LWR model}\label{sec:nonlocal}

The main objective of this paper is to consider the asymptotic stability of a nonlocal  extension to the LWR model \eqref{eq:lwr}, proposed by \cite{Blandin2016,goatin2016well}. The basic assumption underneath such a nonlocal model is that each vehicle perceives traffic density information in a road segment of length $\delta>0$ ahead of the vehicle's current location. The driving speed of the vehicle is then based on an weighted average of density within the road segment:
\begin{align}
	u(x,t)=U\left(\int_0^\delta\rho(x+s,t)w_\delta(s)\,ds\right),\label{eq:nonlocal_relation}
\end{align}
where the nonlocal kernel $w_\delta(\cdot)$ characterizes the nonlocal effect.
\eqref{eq:continuity}\eqref{eq:nonlocal_relation} lead to the following nonlocal LWR model:
\begin{align}
	\partial_t\rho(x,t)+\partial_x\left(\rho(x,t) U\left(\int_0^\delta\rho(x+s,t)w_\delta(s)\,ds\right)\right)=0.\label{eq:nonlocal_lwr}
\end{align}

In the existing studies, some theoretical and numerical results have been developed on the scalar nonlocal conservation law \eqref{eq:nonlocal_lwr}, see Section~\ref{sec:related_work} for a review. 
However, existing studies on the asymptotic stability of the model are still limited.
Under some suitable assumptions, we show that the stability is closely related to the nonlocal kernel $w_\delta(\cdot)$. In particular, {we prove that the solution of the model exponentially converges to a constant density as $t\to\infty$ when the kernel $w_\delta(\cdot)$ is non-increasing and non-constant}.
Meanwhile, a constant kernel may lead to traffic waves that persist in time.

To interpret the significance of the mathematical findings made in this work, let us note their connections to
issues that are  important in real traffic situation.  In the traffic research community,
it is widely recognized that the presence of traffic waves could result in
 elevated risks for traffic safety, an increase in vehicle
fuel consumption, as well as a reduction in total traffic throughput \cite{stern2018dissipation}. Thus, the dissipation of traffic waves and the stability of constant density states are features that can offer benefits to both drivers of individual vehicles and the whole traffic ecosystem.
The particular mathematical results established here,
in plain words, provide further  evidence to the following natural principle when designing driving algorithms for connected vehicles: \emph{it can be beneficial to utilize nonlocal interactions between connected vehicles for traffic decisions, and while doing so, suitable forms and ranges of nonlocality should be adopted with nearby information deserving more attention.}

\subsection{Related work}\label{sec:related_work}

The nonlocal LWR model \eqref{eq:nonlocal_lwr} was first proposed in \cite{Blandin2016,goatin2016well}, where the existence, uniqueness and maximum principle of the weak entropy solution were proved using the Lax-Friedrichs numerical approximation. The entropy condition is adopted to ensure the solution uniqueness.
In subsequent works, \cite{Chiarello2018} proved the same results for a generalized model of \eqref{eq:nonlocal_lwr}:
\begin{align}
	\partial_t\rho(x,t)+\partial_x\left(g(\rho(x,t)) U\left(\int_0^\delta\rho(x+s,t)w_\delta(s)\,ds\right)\right)=0,\label{eq:generalize_1}
\end{align}
and \cite{Chalons2018} developed high-order numerical schemes to solve \eqref{eq:generalize_1}.
In a related work, \cite{keimer2017existence} studied a family of nonlocal balance laws:
\begin{align*}
	\partial_t\rho(x,t)+\partial_x\left(\rho(x,t)U\left(\int_{a(x)}^{b(x)}w_1(x,y,t)\rho(y,t)\,dy\right)\right)=h(x,t),
\end{align*}
which include \eqref{eq:nonlocal_lwr} as a special case. The existence and uniqueness of the weak solution were proved using the method of characteristics and a fixed-point argument. The latter also leads to solution uniqueness without the use of the entropy condition.

In the existing studies, some analytic properties of the nonlocal LWR model \eqref{eq:nonlocal_lwr} were discussed. In terms of solution regularity, \cite{bressan2019traffic} showed that the solution of \eqref{eq:nonlocal_lwr} has the same regularity as the initial data when: (i) the nonlocal kernel $w_\delta(\cdot)$ is $\mathbf{C}^1$ smooth and non-increasing on $[0,+\infty)$ with the zero extension on $[\delta,+\infty)$; (ii) the desired speed function $U(\cdot)$ is $\mathbf{C}^2$ smooth and $U'\leq-c<0$ for some constant $c$. In contrast, the local LWR model \eqref{eq:lwr} can develop shock solutions from smooth initial data whenever the characteristics impinge each other.  
In terms of relations between the local and nonlocal models, one fundamental question is whether the solution of the nonlocal model \eqref{eq:nonlocal_lwr} converges to that of the local model \eqref{eq:lwr} when $\delta\to0$, i.e., the vanishing nonlocality limit. In \cite{colombo2019singular}, it was shown that such convergence is in general false with a demonstration based on an example associated with the desired speed function $U(\rho)=\rho$ and discontinuous initial data. Nevertheless, convergence results were given in \cite{keimer2019approximation} when $U(\cdot)$ is a decreasing function and the initial data is monotone. \cite{bressan2019traffic,bressan2020entropy} considered a special case where the nonlocal interaction range is infinite and the nonlocal kernel is exponential, which leads to 
$$u(x,t)=U\left(\int_0^\infty\delta^{-1} e^{-\frac{s}{\delta}}\rho(x+s,t)\,ds\right).$$
In this case, 
the nonlocal-to-local convergence as $\delta\to0$ was proved for uniformly positive initial data and any desired speed function $U(\cdot)$ that is $\mathbf{C}^2$ smooth and $U'\leq-c<0$ for some constant $c$. \cite{colombo2018blow} extended the convergence results for exponentially decaying kernels and a family of decreasing desired speed functions, but required the initial data to be uniformly positive and have no negative jumps. It is still an open problem what are sufficient and necessary conditions on the desired speed function, nonlocal kernel and initial data for the vanishing nonlocality limit to be true.
In terms of asymptotic behavior, \cite{ridder2019traveling} gave a class of monotone stationary solutions of \eqref{eq:nonlocal_lwr} on an infinitely long road and showed that those solutions are asymptotic local attractors of \eqref{eq:nonlocal_lwr}. \cite{karafyllis2020analysis} studied a generalized model of \eqref{eq:nonlocal_lwr} where a nudging (or ``look behind'') term is added to the nonlocal velocity:
\begin{align*}
	u(x,t)=U\left(\int_0^\delta\rho(x+s,t)w_\delta(s)\,ds\right)\tilde{U}\left(\int_0^{\tilde{\delta}}\rho(x-s,t)\tilde{w}_{\tilde{\delta}}(s)\,ds\right).
\end{align*}
Under the assumptions that: (i) the model is solved on a ring road; (ii) $U(\cdot)$ is decreasing and $\tilde{U}(\cdot)$ is increasing; (iii) $w_\delta(s)=1/\delta$; (iv) $\tilde{\delta}$ is the length of the ring road and $\tilde{w}_{\tilde{\delta}}(s)=(\tilde{\delta}-s)/\tilde{\delta}$, the local exponential stability of uniform equilibrium flows as $t\to\infty$ was proved. 

Let us also briefly mention other relevant studies. The nonlocal LWR model \eqref{eq:nonlocal_lwr} has been generalized to the case for 1-to-1 junctions \cite{chiarello2019junction} and of multi-class vehicles \cite{chiarello2019non,chiarello2020lagrangian}.
There are also nonlocal traffic flow models other than \eqref{eq:nonlocal_lwr}. In \cite{sopasakis2006stochastic}, a model based on Arrhenius ``look-ahead'' dynamics was proposed where the nonlocal velocity:
\begin{align*}
	u(x,t)=U\left(\rho(x,t)\right)\exp\left(-\int_0^\delta \rho(x+s,t)w_\delta(s)\,ds\right),
\end{align*}
\cite{Lee2015,Lee2019a,Lee2019b} analyzed shock formation criteria of the model; In \cite{chiarello2020micro}, a nonlocal extension to the traditional Aw-Rascle-Zhang model \cite{aw2000resurrection} was proposed and the micro-macro limit was demonstrated.
More broadly, nonlocal models have been drawing increasing attention in our connected world \cite{du2019nonlocal}.  Nonlocal conservation laws, in particular, have been studied in many other applications, e.g., pedestrian traffic \cite{Colombo2012,burger2020non}, sedimentation \cite{Betancourt2011} and material flow on conveyor belts \cite{Goettlich2014,rossi2020well}, see \cite{Colombo2016} for a review. 
\cite{du2012new,du2017nonlocal,du2017numerical} discussed nonlocal conservation laws inspired from discrete descriptions of local conservation laws. 
Some more analytical and numerical studies on nonlocal conservation laws can be found in \cite{Aggarwal2015,Amorim2015,Colombo2018,Goatin2019,Chiarello2019,Berthelin2019}.

\subsection{Main results}\label{sec:main_results}

Before the rigorous statement of the main results of this paper, let us specify the set-up of the model problem. First of all, we consider the problem on a ring road. Mathematically, we use the spatial domain $x\in[0,1]$ to represent the ring road and assume the periodic boundary condition for the equation \eqref{eq:nonlocal_lwr}.
\begin{itemize}
	\item[({\bf A1})] {$\rho(0,t)=\rho(1,t)$,\quad $\forall t\geq0$}.
\end{itemize}

The periodicity assumption is common in stability analysis of traffic flow models and fits the scenarios in field experiments \cite{sugiyama2008traffic,stern2018dissipation}. The nonlocal LWR model \eqref{eq:nonlocal_lwr} is solved with the periodic boundary condition and the following initial condition:
\begin{align}
	\rho(x,0)=\rho_0(x),\quad x\in[0,1],\label{eq:ini_data}
\end{align}
where $\rho_0$ is a nonnegative density distribution in $\mathbf{L}^\infty([0,1])$.
We denote:
\begin{align}
	\bar{\rho}=\int_0^1 \rho_0(x)\,dx,\label{eq:rho_bar}
\end{align}
the average density of all vehicles on the ring road. Given $\bar{\rho}$, there is a constant solution of \eqref{eq:nonlocal_lwr}:
\begin{align}
	\rho(x,t)\equiv\bar{\rho}.\label{eq:uniform_sol}
\end{align}
This constant solution, which is an equilibrium of the dynamics described by the nonlocal LWR model \eqref{eq:nonlocal_lwr}, represents the \emph{uniform flow} in traffic where all vehicles are uniformly distributed and drive at the same speed.

We then make the following assumptions on the desired speed function $U(\cdot)$ and the nonlocal kernel $w_\delta(\cdot)$.
\begin{itemize}
	\item[({\bf A2})] $U(\rho)=1-\rho$.
\end{itemize}

The linear desired speed function $U(\rho)=1-\rho$, usually referred to as the Greenshields speed-density relationship \cite{greenshields1935study}, is widely used in traffic flow modeling. We make the assumption ({\bf A2}) to simplify the problem because in this case \eqref{eq:nonlocal_lwr} can be rewritten as:
\begin{align}
	\partial_t\rho(x,t)+\partial_x\left(\rho(x,t)\left(1-\rho(x,t)\right)\right)=\nu(\delta)\partial_x\left(\rho(x,t)\mathcal{D}_x^\delta\rho(x,t)\right),\label{eq:nonlocal_diffusion}
\end{align}
where:
\begin{align}
	\mathcal{D}_x^\delta\rho(x,t)=\frac{1}{\nu(\delta)}\int_0^\delta \left[\rho(x+s,t)-\rho(x,t)\right]w_\delta(s)\,ds,\quad\text{and}\quad\nu(\delta)=\int_0^\delta sw_\delta(s)\,ds.\label{eq:nonlocal_grad}
\end{align}

The equation \eqref{eq:nonlocal_grad} defines the one-sided nonlocal gradient operator $\mathcal{D}_x^\delta$ \cite{dyz17dcdsb,dtty2016cmame}. 
In \eqref{eq:nonlocal_grad}, the integration is defined with respect to 
periodicity of the density function. The formulation \eqref{eq:nonlocal_diffusion} reinterprets the nonlocal LWR model \eqref{eq:nonlocal_lwr} as the local one \eqref{eq:lwr} with an additional  term that may provide some form of nonlocal diffusion for a suitably chosen kernel $w_\delta(\cdot)$. 
A sufficient condition is provided in the following assumption.
 
\begin{itemize}
	\item[({\bf A3})] $w_\delta(\cdot)$ is a $\mathbf{C}^1$ function defined on $[0,\delta]$, satisfying 
	$$w_\delta(s)\geq0, \; \forall s\in[0,\delta]\quad\text{ and }\quad \int_0^\delta w_\delta(s)\,ds=1.$$
	In addition, $w_\delta(\cdot)$ is non-increasing and non-constant on $[0,\delta]$.
\end{itemize}

The assumption ({\bf A3}) is the key to the main findings of this paper. It is the mathematical reformulation of the natural design principle that density information of nearby vehicles should deserve more attention. Under this assumption, we can deduce that the nonlocal LWR model \eqref{eq:nonlocal_lwr} indeed adds appropriate nonlocal diffusion effect to the local one \eqref{eq:lwr} through a direct spectral analysis, see Section~\ref{sec:spectral_estimate}.  
More precisely, we will show the following linear stability result.

\begin{thm}\label{thm:linear}
Under the assumptions \textup{({\bf A1})\,-\,({\bf A3})},
the uniform flow solution defined by \eqref{eq:uniform_sol} is linearly asymptotically stable for any $\bar{\rho}>0$.
\end{thm}

Naturally,  for the nonlinear nonlocal system,
it is interesting to see if we can extend the linear stability to get global nonlinear stability.  For this, we make one additional assumption on the initial data.

\begin{itemize}
	\item[({\bf A4})] There exist $0<\rho_{\text{min}}\leq\rho_{\text{max}}\leq1$ such that:
\begin{align*}
	\rho_{\text{min}}\leq\rho_0(x)\leq\rho_{\text{max}},\quad\forall x\in[0,1].
\end{align*}
\end{itemize}

With all above assumptions, we are ready to state the  well-posedness of the weak solution as defined below.

\begin{defn}\label{defn:weak_sol}
	$\rho\in\mathbf{C}\left([0,\infty);\,\mathbf{L}^1\left([0,1]\right)\right)\cap\mathbf{L}^\infty\left([0,1]\times[0,\infty)\right)$ is a weak solution of \eqref{eq:nonlocal_lwr} with the initial condition \eqref{eq:ini_data} and the periodic boundary condition, if:
	\begin{align*}
		\int_0^\infty\int_0^1\rho(x,t)\partial_t\phi(x,t)+\rho(x,t)U\left(\int_0^\delta\rho(x+s,t)w_\delta(s)\,ds\right)\partial_x\phi(x,t)\,dxdt+\int_0^1\rho_0(x)\phi(x,0)\,dx=0,
	\end{align*}
	for all $\phi\in\mathbf{C}^1\left([0,1]\times[0,\infty)\right)$ periodic in space and having compact support.
\end{defn}

The well-posedness theorem follows from \cite{keimer2019approximation}. Even though the spatial domain considered in that work is set to be the real line $\mathbb{R}$, the same arguments work with little modifications for the periodic case.  

\begin{thm}\label{thm:weak_solution}
	Under the assumptions \textup{({\bf A1})\,-\,({\bf A4})}, the nonlocal LWR model \eqref{eq:nonlocal_lwr} admits a unique weak solution in the sense of Definition~\ref{defn:weak_sol}, and the solution satisfies:
	\begin{align*}
		\rho_{\mathrm{min}}\leq\rho(x,t)\leq\rho_{\mathrm{max}},\quad\forall x\in[0,1],\ t\geq0.
	\end{align*}
\end{thm}

Although the weak solution always exists, it can be discontinuous. In this paper, the stability analysis is based upon an energy estimate. To do that, we make a regularity assumption on the weak solution of \eqref{eq:nonlocal_lwr}.
\begin{itemize}
	\item[({\bf A5})] The weak solution $\rho\in\mathbf{C}^1\left([0,1]\times[0,\infty)\right)$.
\end{itemize}

The assumption ({\bf A5}) is equivalent to say that $\rho$ is the classical solution of \eqref{eq:nonlocal_lwr}. 
When the assumptions ({\bf A1})\,-\,({\bf A4}) are true, \cite{bressan2019traffic} proved a sufficient condition for the assumption ({\bf A5}): the initial data $\rho_0$ is $\mathbf{C}^1$ smooth, and the nonlocal kernel $w_\delta(\cdot)$ is $\mathbf{C}^1$ smooth on $[0,+\infty)$ with the zero extension $w_\delta(s)=0$ on $s\in[\delta,+\infty)$.

Now we are in position to state the main results of this paper.

\begin{thm}\label{thm:main}
	Under the assumptions \textup{({\bf A1})\,-\,({\bf A5})}, and suppose $\rho(x,t)$ is the solution of the nonlocal LWR model \eqref{eq:nonlocal_lwr}. Then there exists a constant $\lambda>0$ that only depends on $\delta$, $w_\delta(\cdot)$ and $\rho_{\mathrm{min}}$, such that:
	\begin{align} \label{eq:main}
		\norm{\rho(\cdot,t)-\bar{\rho}}_{\mathbf{L}^2}\leq e^{-\lambda t}\norm{\rho_0-\bar{\rho}}_{\mathbf{L}^2},\quad\forall t\geq0,
	\end{align}
	where $\bar{\rho}$ is given by \eqref{eq:rho_bar}.
	As a corollary, $\rho(\cdot,t)$ converges to $\bar{\rho}$ in $\mathbf{L}^2\left([0,1]\right)$ as $t\to\infty$.
\end{thm}

\begin{rem}
	Theorem~\ref{thm:main} says that any classical solution of the nonlocal LWR model \eqref{eq:nonlocal_lwr} {converges exponentially to the uniform flow defined by \eqref{eq:uniform_sol}. In other words, the uniform flow is a globally asymptotically stable equilibrium attracting all initial data}. In such a traffic system, traffic waves will dissipate and all vehicles will quickly adjust their moving positions and driving speeds towards the uniform state from any initial traffic conditions. This conclusion is drawn under the  assumptions \textup{({\bf A1})\,-\,({\bf A5})}, which in particular imposes limitations on the nonlocal interactions and other traffic conditions. More discussions in this regard are given in  Section~\ref{sec:discussion}, along with additional estimates on the exponent $\lambda$ in \eqref{eq:main}  to further illustrate  their significance in real traffic scenarios and the design principle for connected vehicles.
\end{rem}

\begin{rem}\label{rem:extension}
Let us  briefly note some possible extensions of Theorem~\ref{thm:main}. First of all, although the regularity assumption ({\bf A5}) is necessary for the energy estimate, it can be removed by considering viscous approximation of the nonlocal LWR model \eqref{eq:nonlocal_lwr} and passing to a vanishing viscosity limit. Secondly, the exponential stability result is also true for the generalized nonlocal LWR model \eqref{eq:generalize_1} with a wide family of functions $g=g(\rho)$. We leave more detailed discussions on these extensions and their interpretations to Section~\ref{sec:discussion}.
\end{rem}

The remainder of this paper is organized as follows: Section~\ref{sec:stab_analysis} is devoted to stability analysis of \eqref{eq:nonlocal_lwr} and the proofs of Theorem~\ref{thm:linear} and Theorem~\ref{thm:main}. Section~\ref{sec:numerical_exp} provides numerical experiments to illustrate the results. Conclusions and future research directions follow in Section~\ref{sec:conclusion}.

\section{Stability analysis}\label{sec:stab_analysis}
This section aims to establish the main stability results stated earlier in Section~\ref{sec:main_results}.
In Section~\ref{sec:spectral_estimate}, we analyze the spectral properties of the nonlocal gradient operator $\mathcal{D}_x^\delta$ and its corresponding nonlocal diffusion operator $\partial_x\mathcal{D}_x^\delta$. The analysis yields the linear stability result and also helps to show the nonlinear stability.
The proof of Theorem~\ref{thm:main} builds on an energy estimate that utilizes two ingredients: a nonlocal Poincare inequality and a Hardy-Littlewood rearrangement inequality. Section~\ref{sec:energy_estimate} derives the energy estimate and Section~\ref{sec:proof} completes the proof of Theorem~\ref{thm:main} based on the two inequalities. {Section~\ref{sec:discussion} discusses further extensions of the theorem and compares the local and nonlocal models}. In Section~\ref{sec:counter_example}, a counterexample is shown that convergence to the uniform flow does not hold when the assumption ({\bf A3}) is not satisfied.

\subsection{Spectral analysis and linear stability}\label{sec:spectral_estimate}

Given the assumption ({\bf A1}), we consider the Fourier series expansion of any real-valued periodic function $\rho(x)$ for $x\in[0,1]$:
\begin{align*}
	\rho(x)=\sum_{k\in\mathbb{Z}} \hat{\rho}(k)e^{2\pi ikx}.
\end{align*}
For the local gradient operator $\partial_x$ and the nonlocal gradient operator $\mathcal{D}_x^\delta$ defined by \eqref{eq:nonlocal_grad}, a straightforward calculation gives:
\begin{align*}
	\partial_x\rho(x)=\sum_{k\in\mathbb{Z}}2\pi ik\hat{\rho}(k)e^{2\pi ikx},\quad\text{and}\quad 
	\mathcal{D}_x^\delta\rho(x)=\sum_{k\in\mathbb{Z}}[ib_\delta(k)+c_\delta(k)]\hat{\rho}(k)e^{2\pi ikx},
\end{align*}
where
\begin{align}
	b_\delta(k)=\frac1{\nu(\delta)}\int_0^\delta \sin(2\pi ks)w_\delta(s)\,ds,\quad \text{and}\quad c_\delta(k)=\frac1{\nu(\delta)}\int_0^\delta [\cos(2\pi ks)-1]w_\delta(s)\,ds.\label{eq:fourier_coeff}
\end{align}
As a corollary, the spectrum of the nonlocal diffusion operator $\partial_x\mathcal{D}_x^\delta$ is given by the discrete set of eigenvalues
$\{-2\pi k b_\delta(k) +  2\pi i k c_\delta(k) \}_{k\in\mathbb{Z}}$. The following lemma gives an estimate on the real parts of the eigenvalues $\{-2\pi k b_\delta(k)\}_{k\in\mathbb{Z}}=\{0\}\cup\{-2\pi k b_\delta(k)\}_{k\geq1}$.

\begin{lemm}\label{lem:spectra}
	Under the assumption \textup{({\bf A3})}, we have:
	\begin{align}
		\alpha\triangleq\inf_{k\geq1}2\pi kb_\delta(k)>0.\label{eq:estimate_alpha}
	\end{align}
\end{lemm}
\begin{proof}
	By \cite[Lemma 2]{du2018stability}, the assumption ({\bf A3}) yields that $b_\delta(k)$ is strictly positive for any $k\geq 1$. In fact, since $w_\delta(\cdot)$ is non-increasing and non-constant,
	 integration by parts gives:
	\begin{align}
		2\pi kb_\delta(k)&=\frac1{\nu(\delta)}\left[w_\delta(0)-w_\delta(\delta)\cos(2\pi k\delta)+\int_0^\delta \cos(2\pi ks)w'_\delta(s)\,ds\right],\notag\\
		&\geq\frac1{\nu(\delta)}\left[w_\delta(0)-w_\delta(\delta)+\int_0^\delta \cos(2\pi ks)w'_\delta(s)\,ds\right]>0,\label{eq:estimate_spectra}
	\end{align}
	for any $k\geq 1$. Meanwhile, when $k\to\infty$, one can apply the Riemann-Lebesgue Lemma to get:
	\begin{align*}
		\liminf_{k\to\infty}2\pi kb_\delta(k)\geq \frac1{\nu(\delta)}\left[w_\delta(0)-w_\delta(\delta)\right]>0.
	\end{align*}
	Combining these facts, we get \eqref{eq:estimate_alpha}.
\end{proof}

We now present the proof of the linear stability given in Theorem~\ref{thm:linear}.

\begin{proof}[Proof of Theorem~\ref{thm:linear}]
To show the linear stability, we simply need to consider the linearized equation of
\eqref{eq:nonlocal_lwr} around the uniform flow $\bar{\rho}$. The equation is given by:
\begin{align}
	\partial_t\tilde{\rho}(x,t)+(1-2\bar{\rho})\partial_x \tilde{\rho}(x,t)
	=\nu(\delta) \bar{\rho} \partial_x \mathcal{D}_x^\delta\tilde{\rho}(x,t).
	\label{eq:nonlocal_diffusion_linearize}
\end{align}
The perturbative solution $\tilde{\rho}(x,t)$ is assumed to have mean zero initially, which remains true for all time.  Hence, for the linear stability, we are concerned with the eigenvalues of the nonlocal diffusion operator $\partial_x\mathcal{D}_x^\delta$ except the single zero eigenvalue with a constant eigenfunction. The real parts of those eigenvalues, as shown in Lemma~\ref{lem:spectra}, are uniformly negative. We thus have the linear stability stated in Theorem~\ref{thm:linear}. 
\end{proof}

\subsection{Energy estimate}\label{sec:energy_estimate}

Suppose $\rho(x,t)$ is any $\mathbf{C}^1$ solution to the nonlocal LWR model \eqref{eq:nonlocal_lwr}. The conservation property gives:
\begin{align*}
	\int_0^1 \rho(x,t)\,dx=\int_0^1\rho_0(x)\,dx=\bar{\rho},\quad\forall t\geq0.
\end{align*}

We define the following \emph{energy function} (aka a Lyapunov functional):
\begin{align}
	E(t)\triangleq\frac12\int_0^1 \left(\rho(x,t)-\bar{\rho}\right)^2\,dx,\quad\forall t\geq0.\label{eq:energy_fun}
\end{align}
It is straightforward to get:
\begin{align}
	\frac{dE(t)}{dt}&=\int_0^1\rho(x,t)\partial_t\rho(x,t)\,dx.\label{eq:tmp_1_1}
\end{align}
Note that \eqref{eq:nonlocal_diffusion} is equivalent to \eqref{eq:nonlocal_lwr}, substituting \eqref{eq:nonlocal_diffusion} into \eqref{eq:tmp_1_1} yields:
\begin{align*}
	\frac{dE(t)}{dt}&=-\int_0^1\rho(x,t)\partial_x\left(\rho(x,t)(1-\rho(x,t))\right)\,dx+\nu(\delta)\int_0^1\rho(x,t)\partial_x\left(\rho(x,t)\mathcal{D}_x^\delta\rho(x,t)\right)\,dx.
\end{align*}
Apply the Newton-Leibniz rule and integration by parts, we obtain:
\begin{align}
	\frac{dE(t)}{dt}=-\nu(\delta)\int_0^1\rho(x,t)\partial_x\rho(x,t)\mathcal{D}_x^\delta\rho(x,t)\,dx.\label{eq:energy_dt}
\end{align}
All boundary terms vanish because of the periodic boundary condition.

\subsection{Proof of Theorem~\ref{thm:main}}\label{sec:proof}
We present two lemmas to estimate the right hand side of \eqref{eq:energy_dt}. Then the conclusion of the theorem follows from the estimate of the energy function $E(t)$.

\begin{lemm}{[Nonlocal Poincare inequality]}\label{lem: nonlocal_poincare}
	Suppose that the nonlocal kernel $w_\delta(\cdot)$ satisfies the assumption \textup{({\bf A3})}. There exists a constant $\alpha>0$ such that for any $\mathbf{C}^1$ periodic function $\rho(x)$ defined on $[0,1]$:
	\begin{align}
	 	\int_0^1 \partial_x\rho(x) \mathcal{D}_x^\delta \rho(x)\,dx\geq \alpha\int_0^1 \left(\rho(x)-\bar{\rho}\right)^2\,dx,\label{eq:nonlocal_poincare}
	\end{align}
	where $\bar{\rho}=\int_0^1\rho(x)\,dx$, $\alpha$ only depends on the nonlocal range $\delta$ and the nonlocal kernel $w_\delta(\cdot)$.
\end{lemm}

\begin{proof}
We have $b_\delta(-k)=-b_\delta(k)$, $c_\delta(-k)=c_\delta(k)$ and $\hat{\rho}(-k)=\overline{\hat{\rho}(k)}$ for all $k\in\mathbb{Z}$.
By Parseval's identity,
\begin{align*}
 	\int_0^1 \partial_x\rho(x) \mathcal{D}_x^\delta \rho(x)\,dx&=\sum_{k\in\mathbb{Z}}-2\pi ik[ib_\delta(k)+c_\delta(k)]|\hat{\rho}(k)|^2=\sum_{k=1}^\infty 4\pi kb_\delta(k)|\hat{\rho}(k)|^2.
\end{align*}
Meanwhile,
\begin{align*}
	\int_0^1 \left(\rho(x)-\bar{\rho}\right)^2\,dx=\sum_{k\neq0}|\hat{\rho}(k)|^2=\sum_{k=1}^\infty2|\hat{\rho}(k)|^2.
\end{align*}
Then the inequality \eqref{eq:nonlocal_poincare} follows from \eqref{eq:estimate_alpha}.
\end{proof}

\begin{rem}
	The nonlocal Poincare inequality \eqref{eq:nonlocal_poincare} generalizes the classical one:
	\begin{align}
		\int_0^1 \left(\partial_x\rho(x)\right)^2\,dx\geq \alpha\int_0^1 \left(\rho(x)-\bar{\rho}\right)^2\,dx,\label{eq:classical_poincare}
	\end{align}
	by introducing the nonlocal gradient operator $\mathcal{D}_x^\delta$. \cite{du2018stability} proposed another generalization of \eqref{eq:classical_poincare}:
	\begin{align}
		\int_0^1 \left(\mathcal{D}_x^\delta \rho(x)\right)^2\,dx\geq \alpha\int_0^1 \left(\rho(x)-\bar{\rho}\right)^2\,dx,\label{eq:corr_poincare}
	\end{align}
	to analyze nonlocal Dirichlet integrals. There, $\mathcal{D}_x^\delta$ uses a symmetric difference quotient, so the eigenvalues of $\mathcal{D}_x^\delta$ only have imaginary parts $b_\delta(k)$. In that case, the singularity of the kernel $w_\delta(\cdot)$ at the origin is necessary to bound $b_\delta(k)$ from below when $k\to\infty$, which then implies \eqref{eq:corr_poincare}. 
	This type of nonlocal Poincare inequality  \eqref{eq:corr_poincare} 
	is further extended in \cite{lee2020nonlocal} where a non-symmetric kernel is used to define the nonlocal gradient, much like the one studied in this work. Then the eigenvalues of $\mathcal{D}_x^\delta$ have both real and imaginary parts. With the kernel $w_\delta(\cdot)$ having no singularity at the origin, the imaginary parts $b_\delta(k)$ decay to zero as $k\to\infty$. However, the real parts $c_\delta(k)$ are bounded from below when $k\to\infty$, thus also leading to \eqref{eq:corr_poincare}.
	
	The inequality \eqref{eq:nonlocal_poincare}, as far as the authors know, has not been presented before. Here, we are estimating the $\mathbf{L}^2$ inner product of $\partial_x\rho$ and $\mathcal{D}_x^\delta\rho$. Although the eigenvalues of $\mathcal{D}_x^\delta$ have both real and imaginary parts because of the non-symmetric kernel, the real parts $c_\delta(k)$ have no contribution to the $\mathbf{L}^2$ inner product. We assume the kernel $w_\delta(\cdot)$ to have no singularity, thus $b_\delta(k)\to0$ when $k\to\infty$. But it does not create any issue since \eqref{eq:nonlocal_poincare} only requires that $2\pi kb_\delta(k)$ is bounded from below. The factor $2\pi k$, which corresponds to the eigenvalues of the local gradient operator $\partial_x$, helps us get the desired result.
\end{rem}

\begin{rem}\label{rem:poincareconstant}
	Let us mention that in some special cases, the nonlocal Poincare inequality \eqref{eq:nonlocal_poincare} can become an equality. For example, when $\delta=1$ and $w_\delta(s)=2(1-s)$, we have:
	\begin{align*}
		\mathcal{D}_x^\delta\rho(x)=6\int_0^1 (1-s)[\rho(x+s)-\rho(x)]\,ds.
	\end{align*}
	A direct calculation gives that $\partial_x\mathcal{D}_x^\delta\rho(x)=-6(\rho(x)-\bar{\rho})-3\partial_x\rho(x)$. That is, the nonlocal diffusion is actually a local term. As a corollary, 
	\begin{align*}
		\int_0^1\partial_x\rho(x)\mathcal{D}_x^\delta\rho(x)\,dx=-\int_0^1(\rho(x)-\bar{\rho})\partial_x\mathcal{D}_x^\delta\rho(x)\,dx=6\int_0^1(\rho(x)-\bar{\rho})^2\,dx,
	\end{align*}
	which is a key ingredient used in \cite{karafyllis2020analysis} to study the nonlocal LWR model with nudging. For more general choices of the nonlocal range $\delta$ and nonlocal kernel $w_\delta(\cdot)$, Lemma~\ref{lem: nonlocal_poincare} provides a more effective way to derive global asymptotic stability as demonstrated in this work.
\end{rem}

A special case of Lemma~\ref{lem: nonlocal_poincare} is when $w_\delta(\cdot)$ is a rescaled kernel: $w_\delta(s)=w_1(s/\delta)/\delta$. That is, the family of kernels $\{w_\delta(\cdot)\}_{\delta\in(0,1]}$ is generated from the kernel $w_1(\cdot)$ defined on $[0,1]$. In this case, it is worthwhile to mention that the constant $\alpha$ in \eqref{eq:nonlocal_poincare} is independent of the nonlocal range $\delta$.

\begin{prop}\label{prop:rescale}
	Suppose $w_\delta(s)=w_1(s/\delta)/\delta$ for all $s\in[0,\delta]$, $\delta\in(0,1]$, where $w_1(\cdot)$ satisfies the
 assumption ({\bf A3}) for $\delta=1$.
	Then there exists a constant $\alpha>0$ only depending on $w_1(\cdot)$ such that for any $\delta\in(0,1]$, \eqref{eq:nonlocal_poincare} holds for the nonlocal kernel $w_\delta(\cdot)$ with the constant $\alpha$.
\end{prop}

\begin{proof}
	It suffices to show $\alpha\triangleq\inf_{k\geq1,0<\delta\leq1}2\pi kb_\delta(k)>0$ where $b_\delta(k)$ is defined in \eqref{eq:fourier_coeff}.

	We denote $a=2\pi k\delta$ and $\nu_1=\int_0^1sw_1(s)\,ds$. Obviously, we have $\nu(\delta)=\delta\nu_1$.
	Then we can rewrite \eqref{eq:estimate_spectra} as:
	\begin{align*}
		2\pi kb_\delta(k)
		&\geq\frac1{\nu_1\delta^2}\left[w_1(0)-w_1(1)+\int_0^1 \cos(as)w'_1(s)\,ds\right],\\
		&=\frac1{\nu_1\delta^2}\int_0^1 [\cos(as)-1]w'_1(s)\,ds.
	\end{align*}
	Note that $w_1(1)<w_1(0)$, there exist constants $0<s_1<s_2<1$ and $\eta>0$, which only depend on $w_1(\cdot)$, such that $w'_1(s)\leq-\eta$ when $s\in[s_1,s_2]$. Hence we have:
	\begin{align*}
		2\pi kb_\delta(k)&\geq\frac\eta{\nu_1\delta^2}\int_{s_1}^{s_2} [1-\cos(as)]\,ds.
	\end{align*}

	When $0<a<1$, we use the inequality $1-\cos(as)\geq\frac{(as)^2}{2}-\frac{(as)^4}{24}\geq\frac{11}{24}(as)^2$ to get:
	\begin{align}
		2\pi kb_\delta(k)\geq\frac{\eta a^2}{\nu_1\delta^2}\cdot\frac{11}{72}(s_2^3-s_1^3)=\frac{11\pi^2k^2\eta}{18\nu_1}(s_2^3-s_1^3)=\alpha_1k^2\geq \alpha_1>0,\label{eq:tmp_2_1}
	\end{align}
	for any $k\geq1$, where the constant $\alpha_1$ only depends on $w_1(\cdot)$.

	When $a\geq1$, consider the following integral as a function of $a$:
	\begin{align*}
		h(a)\triangleq\int_{s_1}^{s_2} [1-\cos(as)]\,ds,\quad a\in[1,+\infty).
	\end{align*}
	Then $h(a)$ is always positive and $h(a)\to s_2-s_1>0$ when $a\to+\infty$.
	Hence $h(a)$ has a lower bound $\alpha_2>0$ for $a\in[1,+\infty)$, and $\alpha_2$ only depends on $w_1(\cdot)$ .
	In this case,
	\begin{align}
		2\pi kb_\delta(k)\geq\frac{\eta\alpha_2}{\nu_1\delta^2}\geq\frac{\eta\alpha_2}{\nu_1}>0.\label{eq:tmp_2_2}
	\end{align}
	The estimates \eqref{eq:tmp_2_1}\eqref{eq:tmp_2_2} give the conclusion.
\end{proof}

Next, we present an inequality similar to \eqref{eq:nonlocal_poincare} to deal with the presence of nonlinearity. Actually, Lemma~\ref{lem:nonlinear_poincare} stated below bridges between linear and nonlinear diffusion in the nonlocal setting. Its proof uses the following Hardy-Littlewood rearrangement inequality on a periodic domain. A similar inequality is used in \cite{bressan2019traffic} to prove that the local limit of nonlocal solutions of \eqref{eq:nonlocal_lwr} satisfies the entropy condition.

\begin{lemm}{[Hardy-Littlewood rearrangement inequality]}\label{lem:hardy-littlewood}
	{Suppose $\rho(x)$ is a continuous periodic function defined on $[0,1]$. For any continuous, monotonically increasing function $f(\cdot)$ and $s\in[0,1]$}:
	\begin{align}
		\int_0^1 f(\rho(x))\rho(x+s)\,dx\leq \int_0^1 f(\rho(x))\rho(x)\,dx.\label{eq:hardy}
	\end{align}
\end{lemm}

\begin{proof}
	We first assume $s\in\mathbb{Q}\cap[0,1]$. Suppose $N$ is a positive integer such that $m=sN$ is a nonnegative integer. Let us consider
	the discrete case with
	\begin{align*}
		0=x_0<x_1<\cdots<x_N=1,\quad x_i=i\Delta x,\quad i=0,\dots,N,
	\end{align*}
	where $\Delta x=1/N$. Denote:
	\begin{align*}
		\rho_i=\rho(x_i),\quad f_i=f(\rho(x_i)),\quad i=0,\dots,N,  \; \text{ with }\; 
		\rho_0=\rho_N, \;  f_0=f_N.
	\end{align*}
	Suppose $\sigma(1),\sigma(2),\dots,\sigma(N)$ is a permutation of $1,2,\dots,N$ such that:
	\begin{align*}
		\rho_{\sigma(1)}\leq\rho_{\sigma(2)}\leq\cdots\leq\rho_{\sigma(N)}. 	
	\end{align*}
	The monotonicity of $f(\cdot)$ yields:
	\begin{align*}
		f_{\sigma(1)}\leq f_{\sigma(2)}\leq\cdots\leq f_{\sigma(N)}. 	
	\end{align*}
	Denote $\tau_m$ the shift permutation defined by $\tau_m(i)=i+m$, $i=1,\dots,N$ (use the circular extension when $i+m>N$). The rearrangement inequality gives:
	\begin{align*}
		\sum_{i=1}^N f_i\rho_{i+m}=\sum_{i=1}^N f_{\sigma(i)}\rho_{\tau_m\circ\sigma(i)}\leq \sum_{i=1}^N f_{\sigma(i)}\rho_{\sigma(i)}=\sum_{i=1}^N f_i\rho_i.
	\end{align*}
	The inequality \eqref{eq:hardy} can then be derived via a limit process. By the density of $\mathbb{Q}\cap[0,1]$ in $[0,1]$, a further limit process can establish \eqref{eq:hardy} for any $s\in [0,1]$.
\end{proof}

\begin{lemm}\label{lem:nonlinear_poincare}
	Suppose that the nonlocal kernel $w_\delta(\cdot)$ satisfies the assumption \textup{({\bf A3})}. For any $\mathbf{C}^1$ periodic function $\rho(x)$ defined on $[0,1]$ and satisfying $\rho(x)\geq\rho_{\mathrm{min}}\geq 0$:
	\begin{align}
	 	\int_0^1 \rho(x)\partial_x\rho(x) \mathcal{D}_x^\delta \rho(x)\,dx\geq \rho_{\mathrm{min}}\int_0^1 \partial_x\rho(x) \mathcal{D}_x^\delta \rho(x)\,dx.\label{eq:nonlinear_poincare}
	\end{align}
\end{lemm}

\begin{proof}
	Define $f(\rho)=\frac12(\rho-\rho_{\text{min}})^2$. Then \eqref{eq:nonlinear_poincare} can be rewritten as:
	\begin{align*}
		\int_0^1\partial_x f(\rho(x))\mathcal{D}_x^\delta\rho(x)\,dx\geq0.
	\end{align*}
	Using integration by parts, it is equivalent to:
	\begin{align}
		\int_0^1 f(\rho(x))\partial_x\mathcal{D}_x^\delta\rho(x)\,dx\leq0.\label{tmp_3_1}
	\end{align}
	We only need to show \eqref{tmp_3_1}. A direct calculation gives:
	\begin{align}
		\partial_x\mathcal{D}_x^\delta\rho(x)&=\frac1{\nu(\delta)}\left[\int_0^\delta \partial_x\rho(x+s)w_\delta(s)\,ds-\partial_x\rho(x)\right],\notag\\
		&=\frac1{\nu(\delta)}\left[\int_0^\delta \partial_s\rho(x+s)w_\delta(s)\,ds-\partial_x\rho(x)\right],\notag\\
		&=\frac1{\nu(\delta)}\left[\rho(x+\delta)w_\delta(\delta)-\rho(x)w_\delta(0)-\int_0^\delta\rho(x+s)w_\delta'(s)\,ds-\partial_x\rho(x)\right].\label{eq:tmp_3_2}
	\end{align}
	We multiply both sides of \eqref{eq:tmp_3_2} by $f(\rho(x))$ and integrate them over the domain $[0,1]$.
	The Newton-Leibniz rule gives:
	\begin{align*}
		\int_0^1f(\rho(x))\partial_x\rho(x)\,dx=0.
	\end{align*}
	Define:
	\begin{align*}
		I(s)\triangleq\int_0^1f(\rho(x))\rho(x+s)\,dx,\quad s\in[0,\delta].
	\end{align*}
	Then we have:
	\begin{align}
		\int_0^1 f(\rho(x))\partial_x\mathcal{D}_x^\delta\rho(x)\,dx&=\frac1{\nu(\delta)}\left[I(\delta)w_\delta(\delta)-I(0)w_\delta(0)-\int_0^\delta I(s)w_\delta'(s)\,ds\right],\nonumber \\
		&=\frac1{\nu(\delta)}\left[\left(I(\delta)-I(0)\right)w_\delta(\delta)+\int_0^\delta \left(I(0)-I(s)\right)w_\delta'(s)\,ds\right].\label{eq:tmp_3_3}
	\end{align}
	When $\rho\geq\rho_{\text{min}}\geq 0$, $f(\rho)$ is monotonically increasing. Lemma~\ref{lem:hardy-littlewood} yields that $I(s)\leq I(0)$ for any $0\leq s\leq\delta$. In addition, the assumption ({\bf A3}) yields that $w_\delta'(s)\leq 0$ for any $0\leq s\leq\delta$. So the both terms on the right hand side of \eqref{eq:tmp_3_3} is non-positive, which gives \eqref{tmp_3_1}.
\end{proof}

Naturally, we will be most interested in applying the above lemma to the case where the density $\rho$ satisfies the assumption ({\bf A4}) so that $\rho_{\text{min}}>0$.

Now we can prove our main results.
\begin{proof}[Proof of Theorem~\ref{thm:main}]
	Define the energy function $E(t)$ by \eqref{eq:energy_fun}. Then the derivative of $E(t)$ is given by \eqref{eq:energy_dt}. By Theorem~\ref{thm:weak_solution}, $\rho(x,t)\geq\rho_{\text{min}}>0$ for all $x\in[0,1]$ and $t\geq0$. Apply Lemma~\ref{lem: nonlocal_poincare} and Lemma~\ref{lem:nonlinear_poincare}, we get the estimate:
	\begin{align*}
		\frac{dE(t)}{dt}\leq-2\nu(\delta)\alpha\rho_{\text{min}}E(t),\quad \forall t\geq0.
	\end{align*}
	By the Gronwall's lemma, $E(t)\leq e^{-2\lambda t}E(0)$ where $\lambda=\nu(\delta)\alpha\rho_{\text{min}}$. It immediately yields the conclusion.
\end{proof}

\begin{rem}\label{rem:kl_divergence}
	An alternative approach to show the exponential stability is to define the Lyapunov functional:
	\begin{align*}
		V(t)=\int_0^1 \rho(x,t)\ln\frac{\rho(x,t)}{\bar{\rho}}\,dx,
	\end{align*}
	which is the Kullback–Leibler divergence from the uniform density to $\rho(x,t)/\bar{\rho}$.
	A calculation similar to that in Section~\ref{sec:energy_estimate} gives:
	\begin{align*}
		\frac{dV(t)}{dt}=-\nu(\delta)\int_0^1\partial_x\rho(x,t)\mathcal{D}_x^\delta\rho(x,t)\,dx.
	\end{align*}
	Apply Lemma~\ref{lem: nonlocal_poincare}, one can get $dV(t)/dt\leq-\nu(\delta)\alpha\int_0^1(\rho(x,t)-\bar{\rho})^2\,dx$.
	The Gronwall's lemma together with the inequality (see \cite{karafyllis2020analysis}):
	\begin{align*}
		\frac1{2\rho_{\mathrm{max}}}\int_0^1(\rho(x,t)-\bar{\rho})^2\,dx\leq V(t)\leq\frac1{2\rho_{\mathrm{min}}}\int_0^1(\rho(x,t)-\bar{\rho})^2\,dx,
	\end{align*}
	gives the exponential convergence in the Kullback–Leibler divergence:
	\begin{align*}
		V(t)\leq e^{-2\lambda t}V(0),
	\end{align*}
	where $\lambda=\nu(\delta)\alpha\rho_{\mathrm{min}}$ and consequently:
	\begin{align}
		\norm{\rho(\cdot,t)-\bar{\rho}}_{\mathbf{L}^2}\leq\left(\frac{\rho_{\mathrm{max}}}{\rho_{\mathrm{min}}}\right)^{\frac12}e^{-\lambda t}\norm{\rho_0-\bar{\rho}}_{\mathbf{L}^2}.\label{eq:exp_stab_alt}
	\end{align}
	The estimate on the Kullback–Leibler divergence was proposed in \cite{karafyllis2020analysis} to prove exponential stability for the nonlocal LWR model with nudging. For the case discussed in this paper, Theorem~\ref{thm:main} provides a sharper result \eqref{eq:main} than \eqref{eq:exp_stab_alt}.
\end{rem}

\subsection{Further discussions on the main results}\label{sec:discussion}

	The energy estimate presented so far requires the regularity assumption ({\bf A5}) on the solution. A natural question is whether Theorem~\ref{thm:main} holds for the general weak solution. To deal with the weak solution, we consider the following viscous nonlocal LWR model:
	\begin{align}
		\partial_t\rho(x,t)+\partial_x\left(\rho(x,t) U\left(\int_0^\delta\rho(x+s,t)w_\delta(s)\,ds\right)\right)=\epsilon\partial_x^2\rho(x,t),\label{eq:viscous_nonlocal_lwr}
	\end{align}
	where $\epsilon>0$ is the viscosity parameter. \cite{colombo2019singular} studied \eqref{eq:viscous_nonlocal_lwr} on the real line and showed the solution well-posedness using a fixed-point theorem and $\mathbf{L}^\infty$ estimates. Based on similar arguments, one can show that \eqref{eq:viscous_nonlocal_lwr} admits a unique weak solution under the conditions in Theorem~\ref{thm:weak_solution}. Further, one can show that the weak solution is $\mathbf{C}^\infty$ smooth by a bootstrap argument. Then we can carry out the energy estimate on \eqref{eq:viscous_nonlocal_lwr} and obtain $E(t)\leq e^{-2(\lambda+c\epsilon) t}E(0)$ where $c>0$ is a constant. Letting $\epsilon\to0$, it can be shown that the solution of \eqref{eq:viscous_nonlocal_lwr} converges to the solution of the original nonlocal LWR model \eqref{eq:nonlocal_lwr} weakly, using similar estimates in \cite{colombo2019singular}. As a corollary, we obtain $E(t)\leq e^{-2\lambda t}E(0)$ since $E(t)$ is a lower semi-continuous functional of $\rho(\cdot,t)$. The conclusion of Theorem~\ref{thm:main} still holds. We only state the extended result, which is the same as Theorem~\ref{thm:main} without the assumption ({\bf A5}).  The detailed proof is skipped.
	
\begin{thm}\label{thm:weak}
	Under the assumptions \textup{({\bf A1})-({\bf A4})}, let $\rho(x,t)$ be the weak solution of the nonlocal LWR model \eqref{eq:nonlocal_lwr}. Then there exists a constant $\lambda>0$ that only depends on $\delta$, $w_\delta(\cdot)$ and $\rho_{\mathrm{min}}$, such that the estimate \eqref{eq:main} holds. As a corollary, $\rho(\cdot,t)$ converges to $\bar{\rho}$, which is given by \eqref{eq:rho_bar}, in $\mathbf{L}^2\left([0,1]\right)$ as $t\to\infty$.
\end{thm}

We now make a couple of additional remarks.

\begin{rem}\label{rem:speed-function}
Let us first  have a discussion on the choice of the speed function. 
The assumption $U(\rho)=1-\rho$ allows us to split the local and nonlocal terms in \eqref{eq:nonlocal_diffusion} and carry out the energy estimate. It is interesting to consider extensions of Theorem~\ref{thm:main} for more general forms of nonlocal velocity selection.
For example, we consider the nonlocal velocity in the following form:
\begin{align}
	u(x,t)=U_0(\rho(x,t))\left(1-\int_0^\delta\rho(x+s,t)w_\delta(s)\,ds\right),\label{eq:reform}
\end{align}
which leads to the generalized nonlocal LWR model \eqref{eq:generalize_1} with $g(\rho)=\rho U_0(\rho)$. When $U_0\equiv1$, it reduces to the case in Theorem~\ref{thm:main}. Based on \eqref{eq:reform}, the generalized model \eqref{eq:generalize_1} can be rewritten as:
\begin{align*}
	\partial_t\rho(x,t)+\partial_x\left(g(\rho(x,t))\left(1-\rho(x,t)\right)\right)=\nu(\delta)\partial_x\left(g(\rho(x,t))\mathcal{D}_x^\delta\rho(x,t)\right).
\end{align*}
\cite{Chiarello2018} proved the same well-posedness results for \eqref{eq:generalize_1} as those for \eqref{eq:nonlocal_lwr} assuming that $g=g(\rho)$ is positive and $\mathbf{C}^1$ smooth. If $g(\rho)$ is bounded away from zero when $\rho\in[\rho_{\mathrm{min}},\rho_{\mathrm{max}}]$, the Hardy-Littlewood rearrangement inequality allows us to remove the nonlinear term $g(\rho)$ and get an estimate similar to \eqref{eq:nonlinear_poincare}. Thus, the conclusion of Theorem~\ref{thm:main} remains true.
\end{rem}

\begin{rem}\label{rem:convergence_speed_delta}
Let us now discuss the exponent $\lambda$ in the exponential decay estimate \eqref{eq:main}.
For a rescaled kernel $w_\delta(s)=w_1(s/\delta)/\delta$ as discussed in Proposition~\ref{prop:rescale},
one can examine how $\lambda$ depends on the kernel $w_1(s)$, the nonlocal range $\delta$ and the initial data. 
The theoretical analysis in the proof of Theorem~\ref{thm:main} gives a lower bound for the exponent $\lambda$ as $\lambda=\nu(\delta)\alpha  \rho_{\mathrm{min}} 
= \delta \nu_1 \alpha  \rho_{\mathrm{min}}$, where $\alpha$ and $\nu_1$ are determined by the kernel $w_1(s)$ and $\rho_{\mathrm{min}}$ is the minimum of initial data. Moreover,  a sharper estimate can be derived for some special cases.
For example, with  the solution (or initial data) sufficiently close to the uniform flow density $\bar{\rho}$, one may replace 
$\rho_{\mathrm{min}}$ by $\bar{\rho}$. Then, for the linear decreasing kernel $w_1(s)=2(1-s)$, similar to calculations carried out earlier, we can get 
\begin{align}
	\lambda=\frac{2}{\delta}\left(1-\frac{\sin(2\pi \delta)}{2\pi \delta}\right)\bar{\rho},\quad \text{for }\;
	\delta\in (0, 1].	
	\label{decay_rate_precise}
\end{align}
For sufficiently small $\delta>0$, the above leads to 
\begin{align}
	\lambda=\frac{4\pi^2}{3}\delta\bar{\rho}.\label{decay_rate_asymptotic}
\end{align}
For numerical validation of these estimates, we refer to Section~\ref{sec:numerical_exp}.
Based on both theoretical estimates and numerical observations, one may consider accelerating the convergence to the uniform flow by increasing the nonlocal range $\delta$, at least in a proper range. In traffic terms, this serves to supplement the design principle presented earlier: while nearby information should be given more attention, within a proper nonlocal range, utilizing information gathered over a wider domain 
could bring more benefits. However, as $\lambda$ may not stay monotonically increasing for all $\delta$, the acceleration might become
less effective if $\delta$ gets too large. Thus, one should choose suitably 
the range of nonlocal information to be utilized. Meanwhile, it should be noted that 
the value of $\delta$ should also be properly confined in practice to avoid any significant deviation of each vehicle's driving speed from its desired local speed in consideration of driving safety.
Likewise, we can also see from the above estimates that  the convergence gets faster with larger values of  $\bar{\rho}$. We can attribute this property to the nonlinear dependence of the diffusion introduced to the system \eqref{eq:nonlocal_diffusion} on the traffic density.
\end{rem}

Finally, let us make some comparisons between the nonlocal LWR model \eqref{eq:nonlocal_lwr} and the local one \eqref{eq:lwr}. 
In particular, in terms of the practical implication on traffic flows, it is interesting to examine the rate at which the traffic density would get back to the uniform state.
On one hand, as shown in Theorem~\ref{thm:main}, the solution of \eqref{eq:nonlocal_lwr} has an exponential convergence towards the uniform flow. On the other hand, \cite{debussche2009long} showed that every solution of \eqref{eq:lwr} with the periodic boundary condition converges to the uniform flow $\bar{\rho}$ given by \eqref{eq:rho_bar} as $t\to\infty$, except the case $\bar{\rho}=0.5$ in which the local flux $f(\rho)=\rho(1-\rho)$ in \eqref{eq:lwr} is degenerate.  
However, the convergence will be much slower than the exponential convergence of the nonlocal LWR model \eqref{eq:nonlocal_lwr}. We will demonstrate the asymptotic convergence speed of the local LWR model \eqref{eq:lwr} in the following example. 

Consider the equation \eqref{eq:lwr} with the following linear initial data:
\begin{align}
	\rho_0(x)=\beta x,\quad x\in[0,1],\label{eq:linear_ini}
\end{align}
where $\beta\in[0,1]$ is a constant. In this case, \eqref{eq:lwr} can be solved explicitly and the solution is a piecewise linear function. When $0\leq t\leq\frac1{2\beta}$, there is a moving rarefaction wave and the solution is give by:
\begin{align*}
	\rho(x,t)=\begin{dcases}
	    \frac{t-x}{2t},\quad &(1-2\beta)t\leq x<t;\\
	    \frac{\beta(x-t)}{1-2\beta t},\quad &t\leq x<(1-2\beta)t+1.
	\end{dcases}
\end{align*}
When $t>\frac1{2\beta}$, the solution develops a shock wave. The shock wave moves at a constant speed $1-\beta$ and has a jump from $\rho_l(t)=\frac\beta2-\frac1{4t}$ to $\rho_r(t)=\frac\beta2+\frac1{4t}$.
Before the shock formation ($t\leq\frac1{2\beta}$), the $\mathbf{L}^2$ error $\norm{\rho(\cdot,t)-\bar{\rho}}_{\mathbf{L}^2}$ between the solution and the uniform flow is constant; After the shock formation ($t>\frac1{2\beta}$), a direct calculation gives:
\begin{align}
	\norm{\rho(\cdot,t)-\bar{\rho}}_{\mathbf{L}^2}=\frac{1}{2\sqrt{12}t}.\label{eq:local_linear_convergence}
\end{align}
That is, the solution converges to the uniform flow when $t\to\infty$ with algebraic decay rate, which means that it would take much more time in the local case for the traffic to get to the uniform flow than that predicted in the nonlocal case.

\subsection{A counterexample with the constant kernel}\label{sec:counter_example}

Suppose $\delta=1/m$ where $m$ is a positive integer. We pick the constant kernel:
\begin{align}
	w_\delta(s)=\frac1\delta,\quad s\in[0,\delta].\label{eq:constant_kernel}
\end{align}
Suppose that the initial data $\rho_0(x)$ is periodic with period $\delta$, the solution of the nonlocal LWR model \eqref{eq:nonlocal_lwr} can be explicitly given by $\rho(x,t)=\rho_0(x-\bar{u}t)$ where $\bar{u}=1-\bar{\rho}$. Note that, at any time $t$, the density $\rho(\cdot,t)$ is a translation of $\rho_0$ and hence periodic with period $\delta$. Consequently, the velocity:
\begin{align*}
	u(x,t)=1-\int_0^\delta\rho(x+s,t)w_\delta(s)\,ds=1-\frac1\delta\int_0^\delta\rho(x+s,t)\,ds=1-\bar{\rho}=\bar{u},
\end{align*}
is a fixed constant. The nonlocal LWR model \eqref{eq:nonlocal_lwr} then becomes a scalar transport equation:
\begin{align*}
	\rho_t+\bar{u}\rho_x=0,
\end{align*}
whose solution is the traveling wave $\rho(x,t)=\rho_0(x-\bar{u}t)$. The traveling wave solution never converges to the uniform flow as $t\to\infty$ unless $\rho_0(x)$ is constant. The same form of counterexample was also proposed in \cite{karafyllis2020analysis} with $\delta\in\mathbb{Q}$ and $\rho_0(x-\bar{u}t)$ being a sine-wave.

This counterexample justifies the key role of the assumption ({\bf A3}). When $\rho_0(x)$ is $\mathbf{C}^1$ smooth and bounded between $\rho_{\text{min}}>0$ and $\rho_{\text{max}}\leq1$, so is the solution. All assumptions in Theorem~\ref{thm:main} are satisfied expect ({\bf A3}). However, the conclusion of the theorem fails to be true because we no longer have the nonlocal Poincare inequality \eqref{eq:nonlocal_poincare} for the constant kernel. To wit, we calculate eigenvalues of the nonlocal gradient operator $\mathcal{D}_x^\delta$ with the constant kernel defined in \eqref{eq:constant_kernel}. For the eigenfunction $e^{2\pi imx}$ with frequency $m=1/\delta$, the real part of the corresponding eigenvalue is:
\begin{align*}
	b_\delta(m)=\frac1{\delta\nu(\delta)}\int_0^\delta \sin(2\pi ms)\,ds=0,
\end{align*}
which makes $\alpha=0$ in \eqref{eq:estimate_alpha}. 
In other words, the properties of the nonlocal kernel $w_\delta(\cdot)$ are essential to guarantee that the nonlocal term in \eqref{eq:nonlocal_diffusion} adds appropriate diffusion effect to dissipate traffic waves.

\section{Numerical experiments}\label{sec:numerical_exp}

In this section, we present results of numerical experiments to  further illustrate the established findings and to explore 
cases not covered by the theoretical results.
The following models are considered.
\begin{itemize}
	\item The local LWR \eqref{eq:lwr};
	\item The nonlocal LWR \eqref{eq:nonlocal_lwr} with the linear decreasing kernel $w_\delta(s)=2(\delta-s)/\delta^2$;
	\item The nonlocal LWR \eqref{eq:nonlocal_lwr} with the constant kernel $w_\delta(s)=1/\delta$.
\end{itemize}

All three models are solved by the Lax-Friedrichs scheme with spatial mesh size $\Delta x=2\times10^{-4}$. For more details about this numerical scheme applying to the nonlocal LWR, see \cite{goatin2016well}. To visualize the evolution of traffic densities solved from the models, we plot their snapshots at selected times, with different colors. Furthermore, we compare asymptotic convergence speeds of the solutions to the uniform flow by plotting the $\mathbf{L}^2$ error $\norm{\rho(\cdot,t)-\bar{\rho}}_{\mathbf{L}^2}$ as a function of time $t$. We present these convergence speed plots on different time scales, that is, the semi-log plot to represent the cases with an exponential decay in time, i.e., $\norm{\rho(\cdot,t)-\bar{\rho}}_{\mathbf{L}^2}\propto e^{-\lambda t}$ for some $\lambda>0$, and the log-log plot for cases showing only an algebraic decay in time, in particular, $\norm{\rho(\cdot,t)-\bar{\rho}}_{\mathbf{L}^2}\propto 1/t$. We also remark that, with the linear decreasing kernel, the exponent $\lambda$ of the exponential decay rate can be estimated theoretically by \eqref{decay_rate_precise}
and \eqref{decay_rate_asymptotic}. These theoretical estimates are compared with the value of $\lambda$ estimated from the numerical solutions. 

{\bf Experiment 1.} The first experiment aims to validate the quick dissipation of traffic waves established in Theorem~\ref{thm:main}. In this experiment, we choose a bell-shape initial data:
\begin{align*}
	\rho_0(x)=0.4+0.6\exp\left(-100(x-0.5)^2\right).
\end{align*}
It represents the scenario that initially vehicles cluster near $x=0.5$ and the traffic is lighter in other places. 
We compare solutions of three models solved with the initial data. The results are plotted in Figure~\ref{fig:bellshape}.

For the local LWR, the solution first develops a shock wave from the smooth initial data. Then the shock wave dissipates at a speed no faster than the algebraic decay.  At time $t=6$, the $\mathbf{L}^2$ error $\norm{\rho(\cdot,t)-\bar{\rho}}_{\mathbf{L}^2}$ is on the scale of $10^{-2}$ and one can still visually observe a jump in density at the shock.

For the nonlocal LWR, the nonlocal range is set as $\delta=0.2$.
With the linear decreasing kernel, the solution remains smooth and the initial high density near $x=0.5$ quickly dissipates. The solution converges to the uniform flow exponentially with the numerically estimated exponent $\lambda=1.26$, which is very close to the theoretical estimate $\lambda=1.23$ given by \eqref{decay_rate_precise}.
 At time $t=6$, the whole density profile is nearly uniform with the $\mathbf{L}^2$ error on the scale of $10^{-4}$.
With the constant kernel, the solution first has an exponential convergence to the uniform flow, but the $\mathbf{L}^2$ error stagnates on the scale of $10^{-3}$ after $t=2.5$, which means that there are non-dissipative traffic waves with small amplitudes. The contrast between the case using a linear decreasing kernel and that with a constant kernel
helps to illustrate the 
natural design principle concerning the use of nonlocal information, that is, placing more attention on the nonlocal density information of nearby vehicles could result in better traffic conditions.

{\bf Experiment 2.} The second experiment aims to check the case with linear initial data as discussed in Section~\ref{sec:discussion}. In this experiment, we choose the initial data to be a linear function as in \eqref{eq:linear_ini} with $\beta=0.5$.
We compare solutions of three models solved with the initial data.
The results are plotted in Figure~\ref{fig:linear}.

For the local LWR, the solution is a piecewise linear function as given in Section~\ref{sec:discussion}. At $t=\frac1{2\beta}=1$, a shock wave forms. Then the shock wave dissipates and the solution converges to the uniform flow with an algebraic decay in the $\mathbf{L}^2$ error, i.e., $\norm{\rho(\cdot,t)-\bar{\rho}}_{\mathbf{L}^2}=a/t$ where the estimated value of the coefficient is $a=0.142$. The result validates the analytically derived value $a=\frac{1}{2\sqrt{12}}=0.1443\dots$  based on the estimate in \eqref{eq:local_linear_convergence}.

For the nonlocal LWR, the nonlocal range is set as $\delta=0.2$. 
With the linear decreasing kernel, the exponential convergence to the uniform flow is observed with the numerically estimated exponent $\lambda=0.66$, also effectively predicted by the theoretical estimate $\lambda=0.61$ given by \eqref{decay_rate_precise}. 
Meanwhile, we also observe that the traffic density is no longer piecewise linear when $t>0$. A complicated  dynamic process is involved in the transition from the linear initial data to the uniform flow. This shows that the nonlocal LWR may have richer transient behaviors than the local LWR. 
With the constant kernel, similar patterns are observed but the $\mathbf{L}^2$ error stagnates on the scale of $10^{-2}$ after $t=2.5$ because of the existence of non-dissipative traffic waves.

{\bf Experiment 3.} The third experiment aims to validate the counterexample given in Section~\ref{sec:counter_example}. In this experiment, we choose the initial data to be a sine-wave:
\begin{align*}
	\rho_0(x)=0.5+0.4\sin(4\pi x),
\end{align*}
which is periodic with a period $0.5$. We compare two solutions with the initial data: one solved from the nonlocal LWR with the linear decreasing kernel, the other solved from that with the constant kernel. For both cases, the nonlocal range is set to be the same as the period of the initial data, i.e., $\delta=0.5$.
The results are plotted in Figure~\ref{fig:sinewave}.

With the linear decreasing kernel, the sine wave quickly dissipates and the solution converges to the uniform flow exponentially with the numerically estimated exponent $\lambda=2.02$. In comparison, the theoretical estimate \eqref{decay_rate_precise} gives $\lambda=2.00$ while the asymptotic estimate \eqref{decay_rate_asymptotic} gives $\lambda=3.29$. The result shows that the the actual exponent $\lambda$ may deviate away from the linear relation described by \eqref{decay_rate_asymptotic} when $\delta$ is large but can still be effectively predicted by \eqref{decay_rate_precise}.

With the constant kernel, the solution is a traveling wave moving at the constant speed $\bar{u}=1-\bar{\rho}=0.5$ and the $\mathbf{L}^2$ error stays constant in time and never decays. In this case, vehicles need to repeatedly accelerate and decelerate in accordance with the oscillations in traffic density, resulting in a worse traffic situation even
in comparison to that modeled by the local LWR.   This again reinforces the advantage of paying more attention to the nearby density 
information  when nonlocal information is utilized.

{\bf Experiment 4.} The fourth experiment aims to examine the impact of the nonlocal range $\delta$. In this experiment, we focus on the nonlocal LWR with the linear decreasing kernel, and choose a piecewise constant initial data:
\begin{align*}
	\rho_0(x)=\begin{dcases} 0.25,\quad 0\leq x<0.5;\\0.75,\quad 0.5\leq x<1.\end{dcases}
\end{align*}
For various values of the nonlocal range $\delta$, solutions of the model with the initial data are compared. 
The results are plotted in Figure~\ref{fig:piece_const}.

In the first row of Figure~\ref{fig:piece_const}, we plot traffic density evolution for the solutions with $\delta=0.1$ and $\delta=0.2$. Although the initial data is discontinuous, the dissipation of traffic waves and the convergence to the uniform flow can still be observed.
This result indicates that the regularity assumption in Theorem~\ref{thm:main} might not be necessary, as discussed in Section~\ref{sec:discussion}.

In the bottom left figure of Figure~\ref{fig:piece_const}, we compare the decay rates of convergence for the solutions with $\delta$ ranging from $0.1$ to $0.3$ with a step $0.05$. The result shows that all solutions have exponential convergence to the uniform flow for these values of $\delta$, while the convergence is faster with a larger $\delta$.
The exponent $\lambda$ numerically estimated from the solutions with these values of $\delta$ are compared with the theoretical estimates given by \eqref{decay_rate_precise}, as shown in the bottom right figure of Figure~\ref{fig:piece_const}. One can observe an effective match between the theoretical and numerical estimates.

{\bf Experiment 5.} Finally, let us examine how the mean density affects the decay rate of convergence.
In this experiment, we focus on the nonlocal LWR with the linear decreasing kernel. The nonlocal range is fixed to be $\delta=0.2$. We choose a family of bell-shape initial data:
\begin{align*}
	\rho_0(x)=\rho_{\text{min}}+0.6\exp\left(-100(x-0.5)^2\right),
\end{align*}
with $\rho_{\text{min}}\in[0, 0.4]$. Such a family of initial data have the same variation but different mean densities. We compare the decay rates of convergence for the solutions with $\rho_{\text{min}}$ ranging from $0$ to $0.4$ with a step $0.1$, as shown on the left side of Figure \ref{fig:ini}.

We first observe that even in the case with $\rho_{\text{min}}=0$, meaning that the initial data is supported on a subinterval of the domain and vanishes outside, the solution still converges to the uniform flow exponentially.
This example shows the possibility that the established global stability result may still be true for nonnegative initial data with positive mean densities, which presents an interesting problem to be further studied theoretically in the future.
Moreover, we observe that the convergence becomes faster as the mean density of the initial data gradually increases, again consistent with the theoretical findings discussed earlier.  To better capture the dependence of the exponent $\lambda$ on the mean density $\bar{\rho}$, we do a linear fitting 
of the numerically estimated values of $\lambda$ with respect to $\bar{\rho}$, see the plot on the right side of Figure~\ref{fig:ini}. The result shows that a linear relation of the form $\lambda=2.53\bar{\rho}$ can effectively describe the dependence of $\lambda$ on $\bar{\rho}$. 
As a comparison, the theoretical estimate \eqref{decay_rate_precise} gives $\lambda=2.43\bar{\rho}$, which is also shown on the right side of Figure~\ref{fig:ini}. We observe that the theoretical estimate effectively matches numerical observations. In addition, the largest deviation of the numerical estimate of $\lambda$ from both the linear fitting and the theoretical estimate occurs when $\rho_{\text{min}}=0$. This is an interesting phenomenon indicating that there might be a sharper estimate with the existence of vacuum densities in initial data. 

\begin{figure}[htbp]
  \centering
  \begin{subfigure}{.32\textwidth}
    \includegraphics[width=\textwidth]{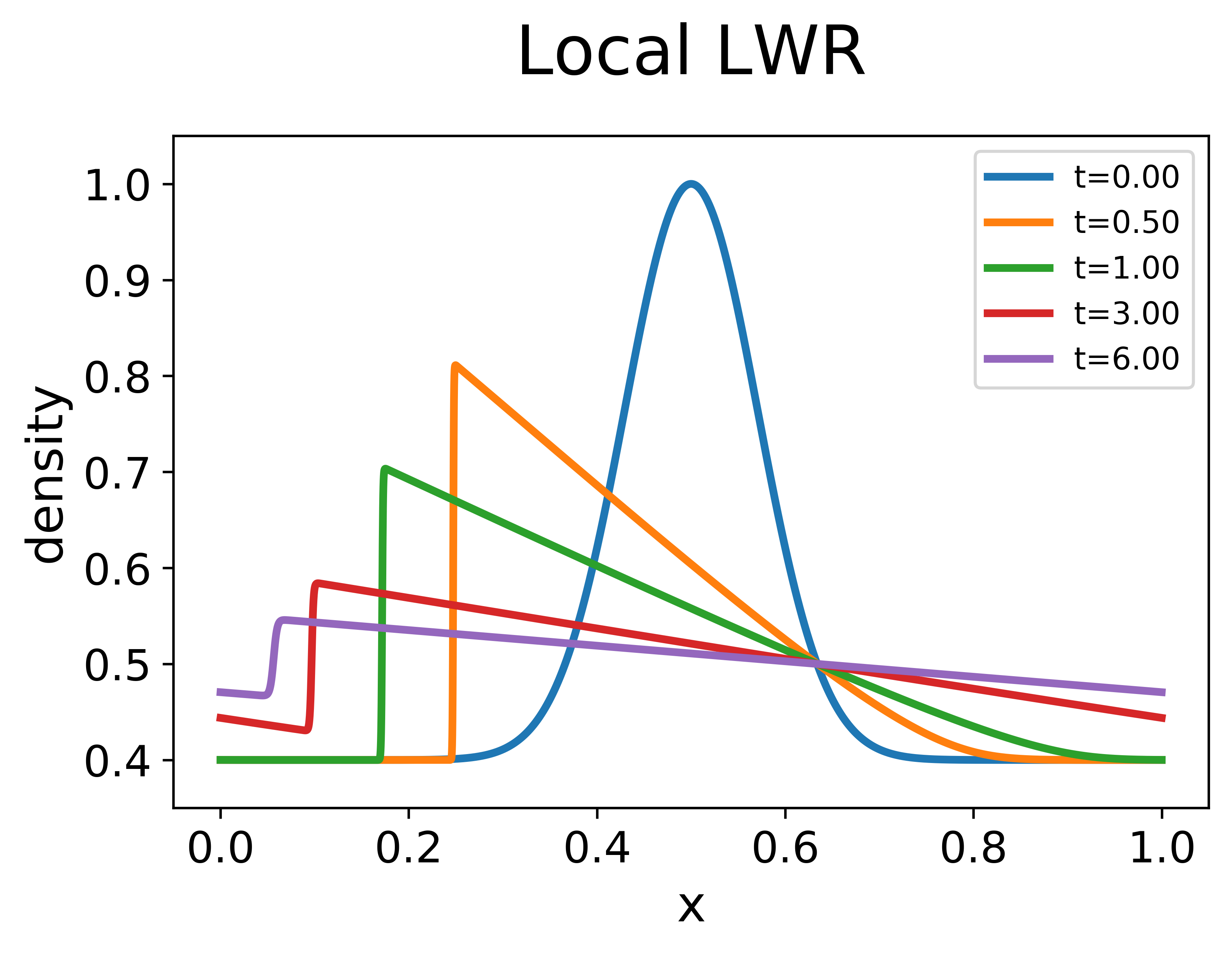}
  \end{subfigure}
  \begin{subfigure}{.32\textwidth}
    \includegraphics[width=\textwidth]{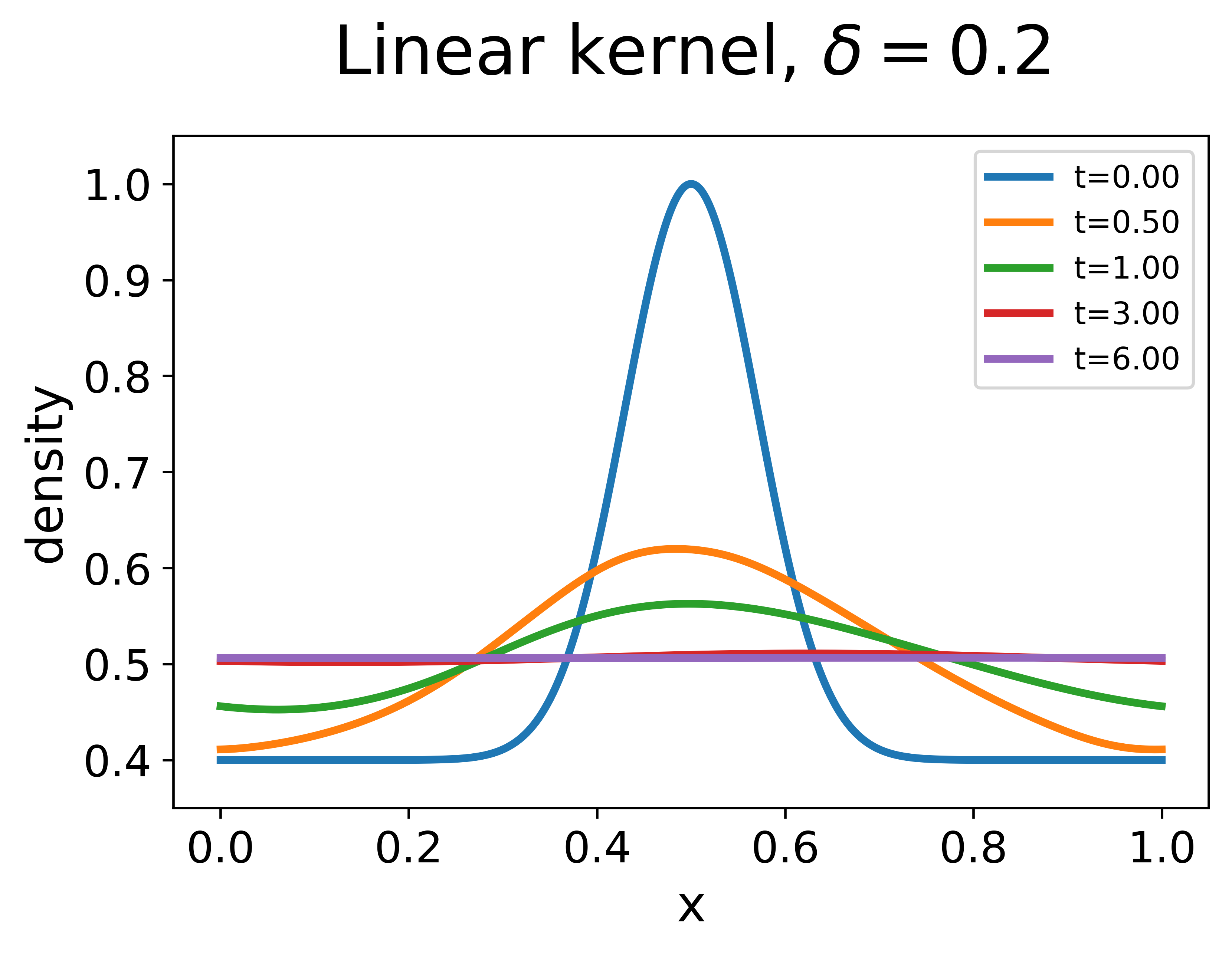}
  \end{subfigure}
  \begin{subfigure}{.32\textwidth}
    \includegraphics[width=\textwidth]{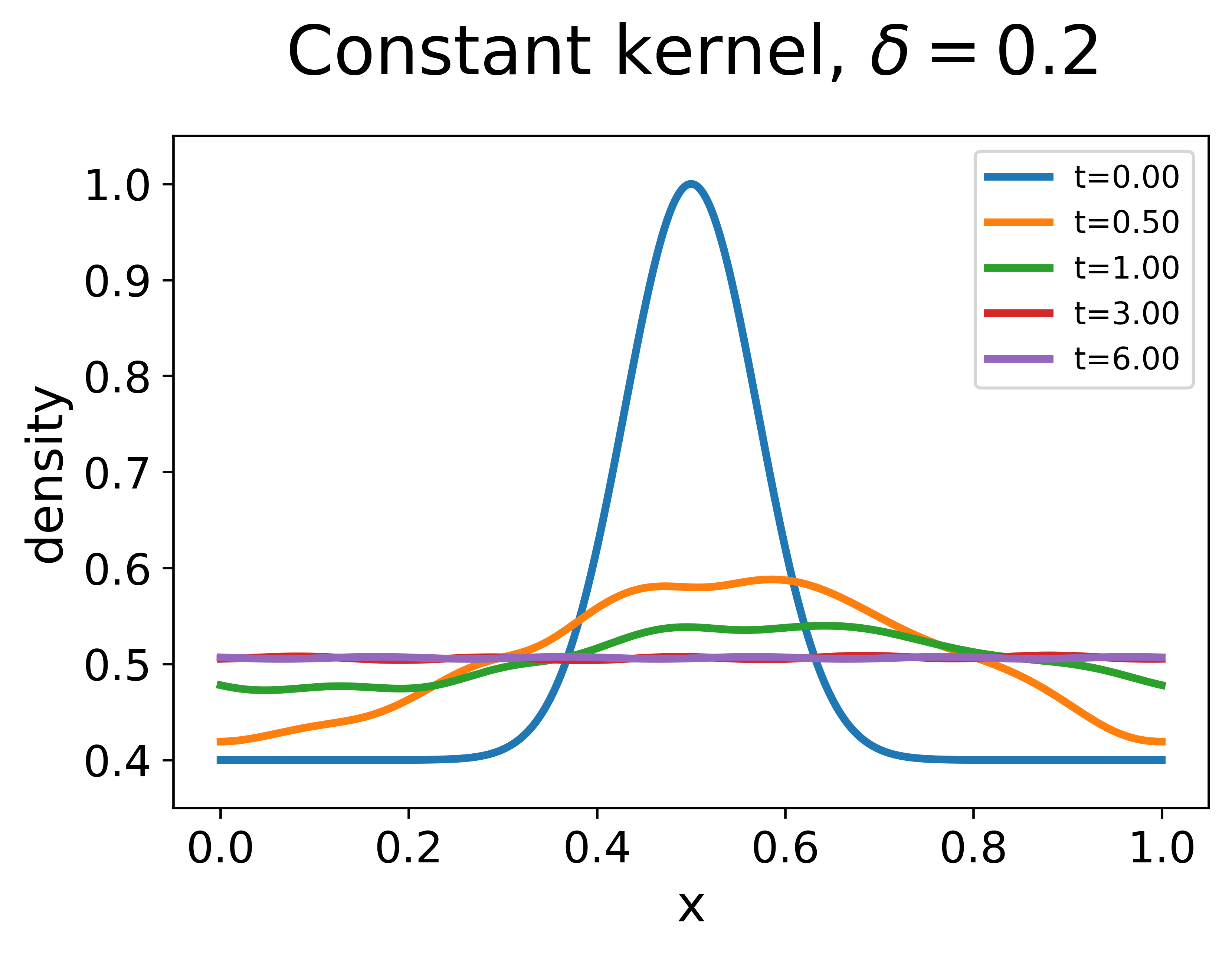}
  \end{subfigure}
  \begin{subfigure}{.32\textwidth}
    \includegraphics[width=\textwidth]{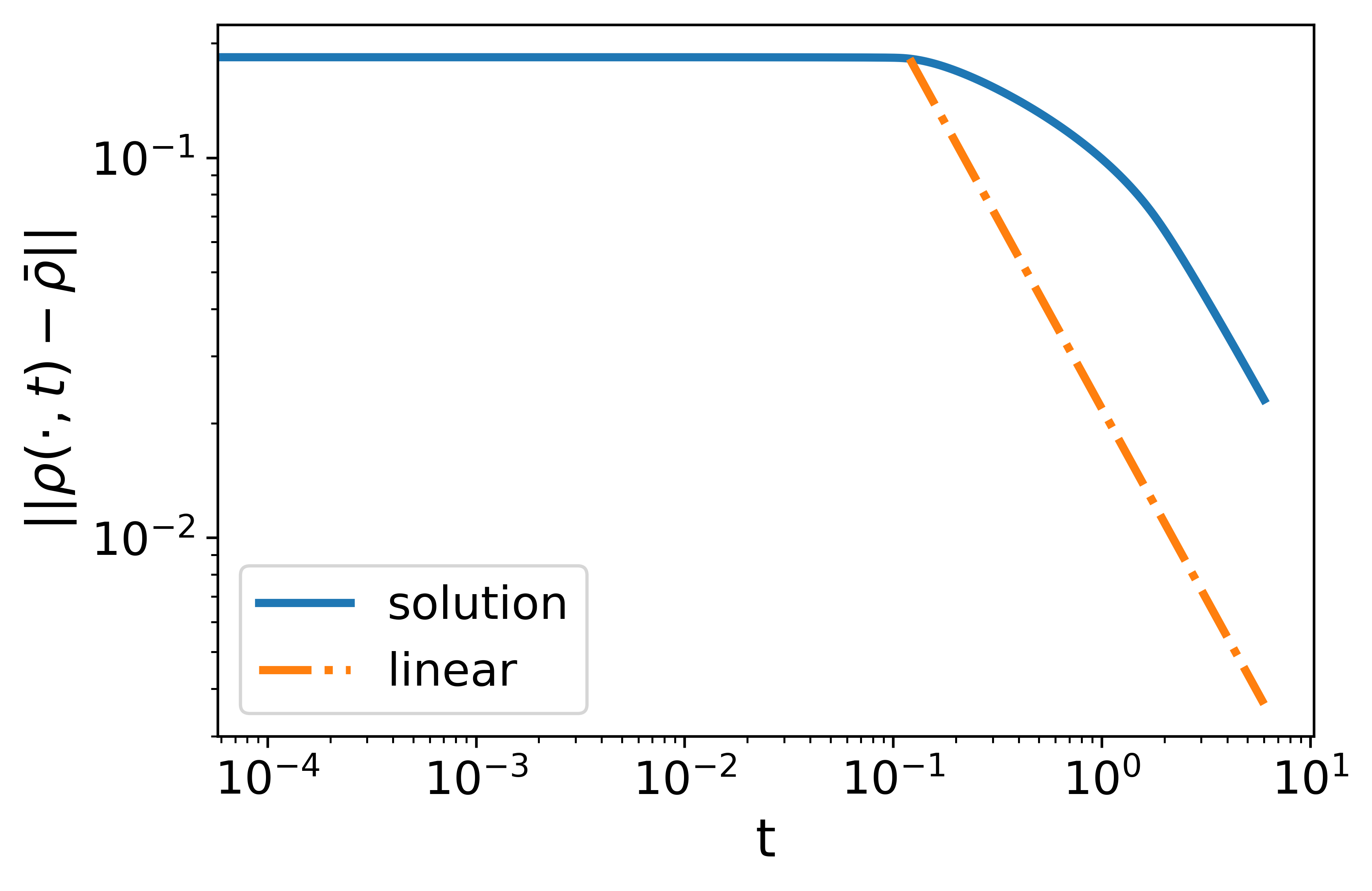}
  \end{subfigure}
  \begin{subfigure}{.32\textwidth}
    \includegraphics[width=\textwidth]{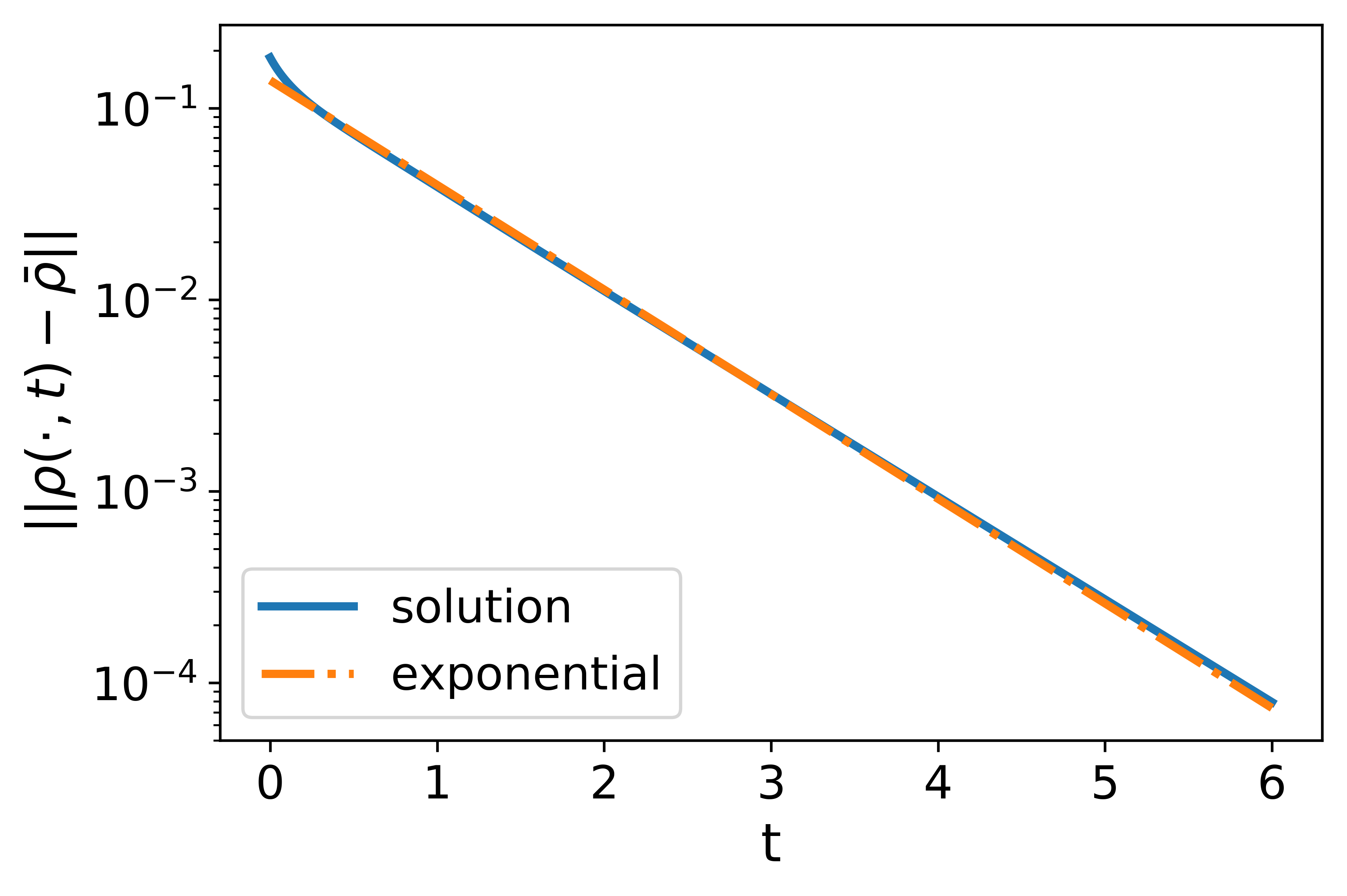}
  \end{subfigure}
  \begin{subfigure}{.32\textwidth}
    \includegraphics[width=\textwidth]{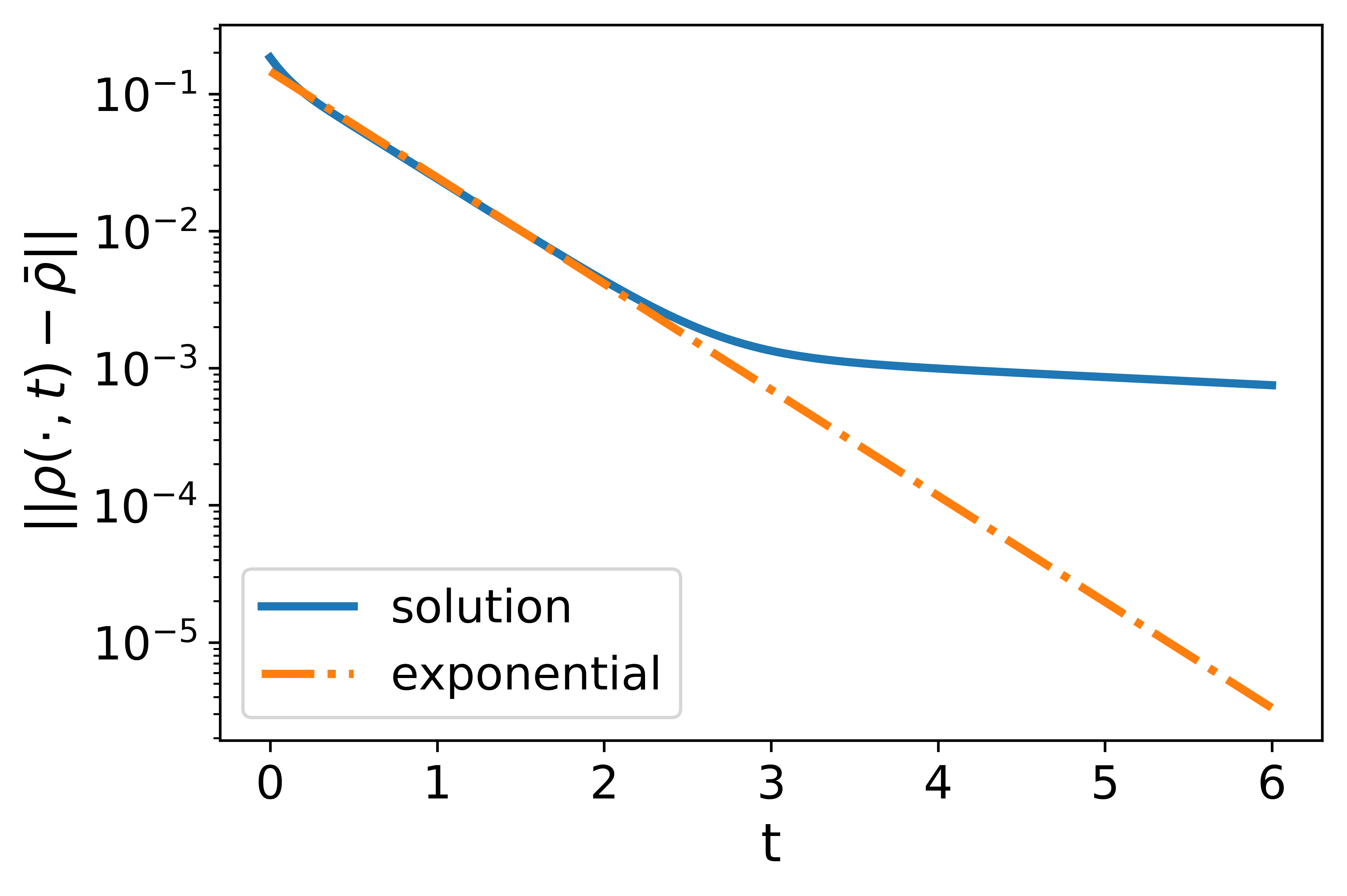}
  \end{subfigure}
  \caption{Compare solutions from different models with the bell-shape initial data. 
  }
  \label{fig:bellshape}
\end{figure}

\begin{figure}[htbp]
  \centering
  \begin{subfigure}{.32\textwidth}
    \includegraphics[width=\textwidth]{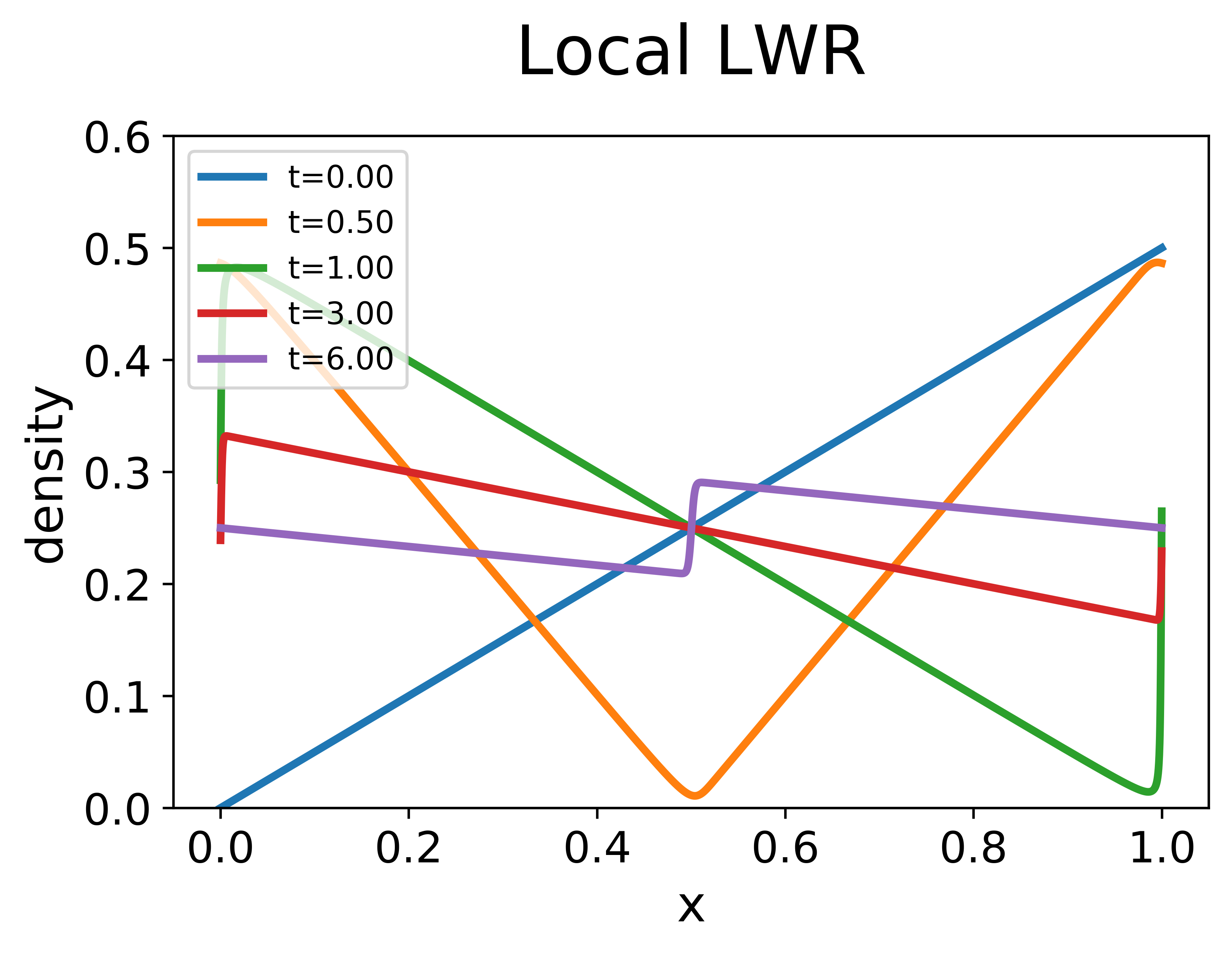}
  \end{subfigure}
  \begin{subfigure}{.32\textwidth}
    \includegraphics[width=\textwidth]{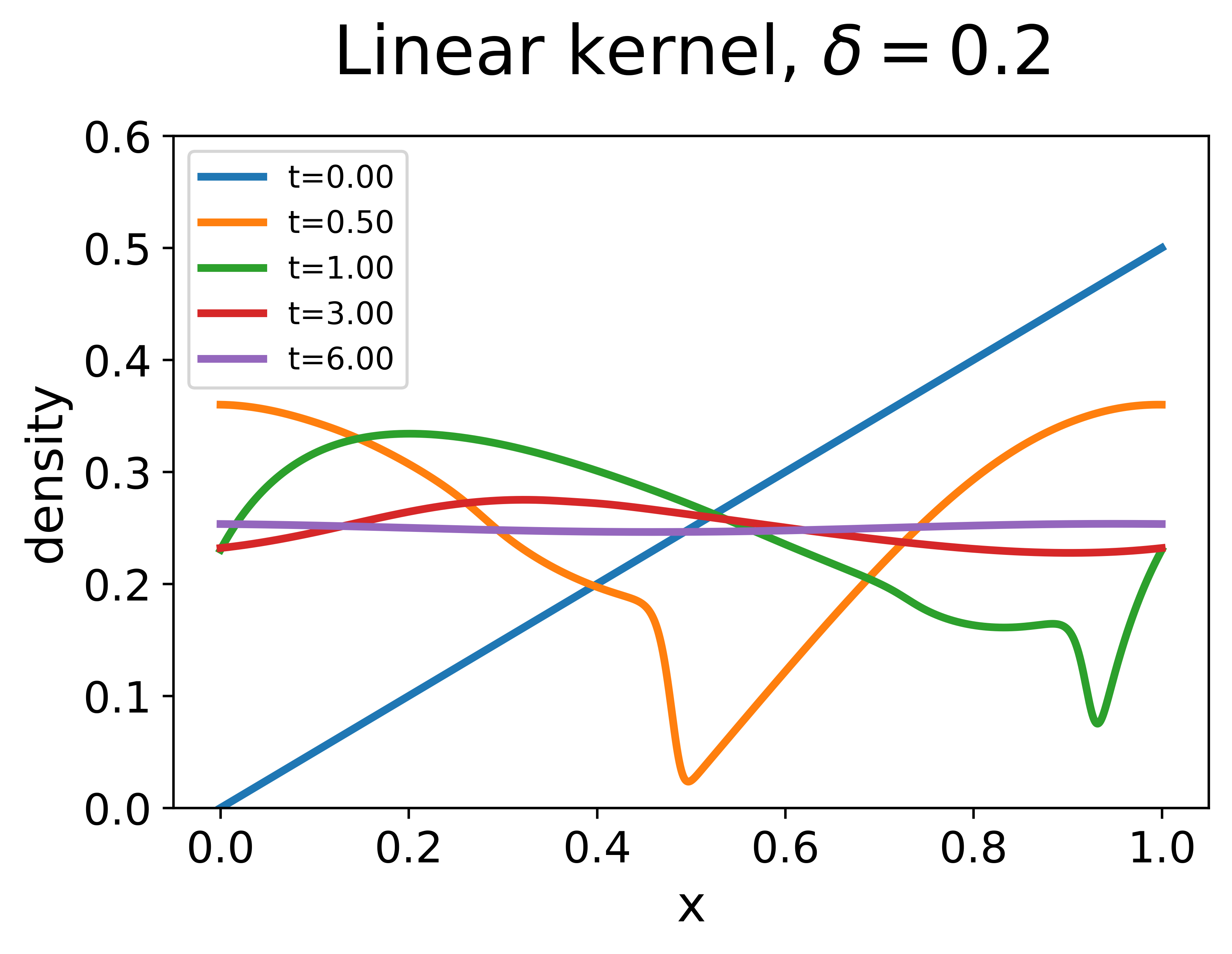}
  \end{subfigure}
  \begin{subfigure}{.32\textwidth}
    \includegraphics[width=\textwidth]{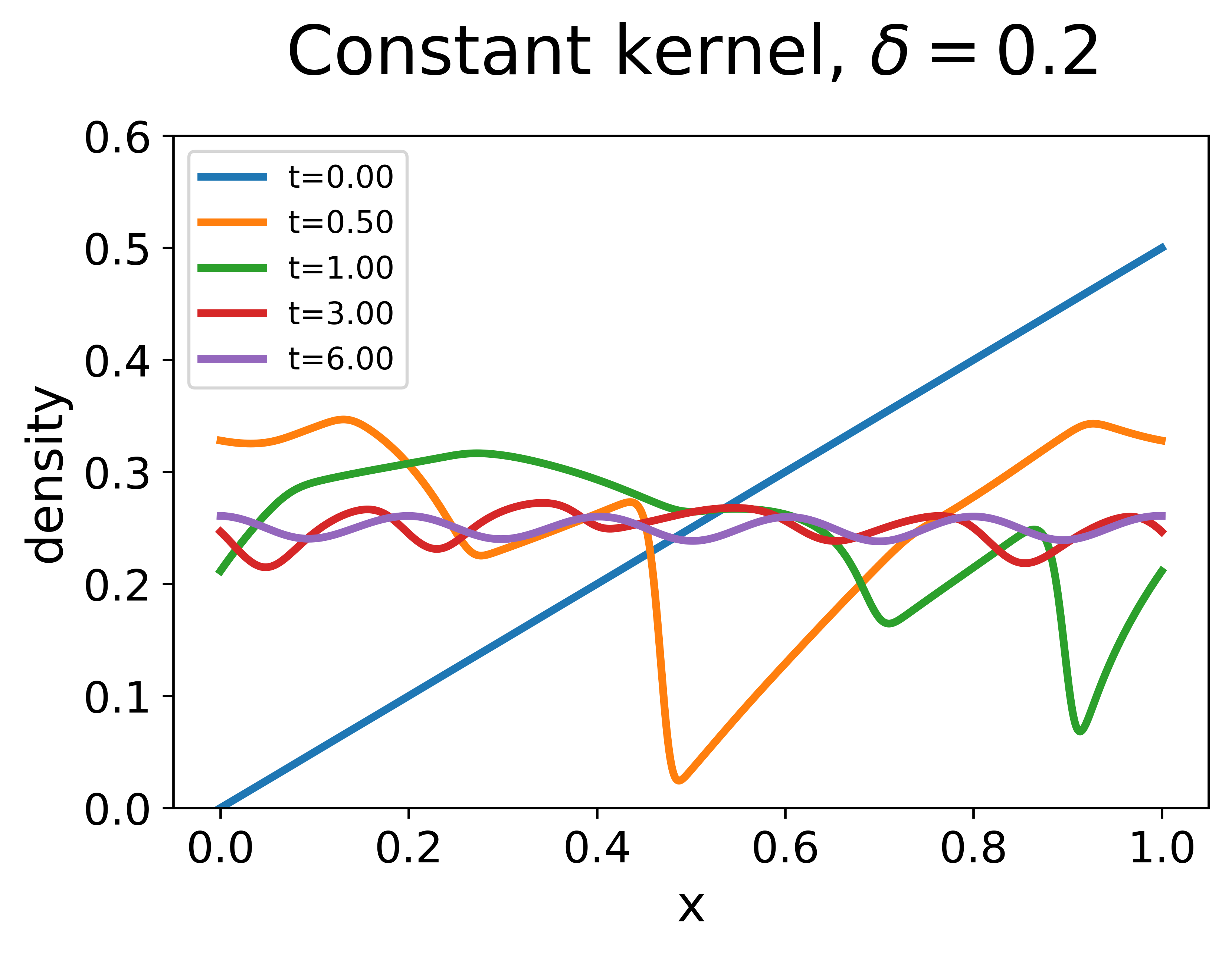}
  \end{subfigure}
  \begin{subfigure}{.32\textwidth}
    \includegraphics[width=\textwidth]{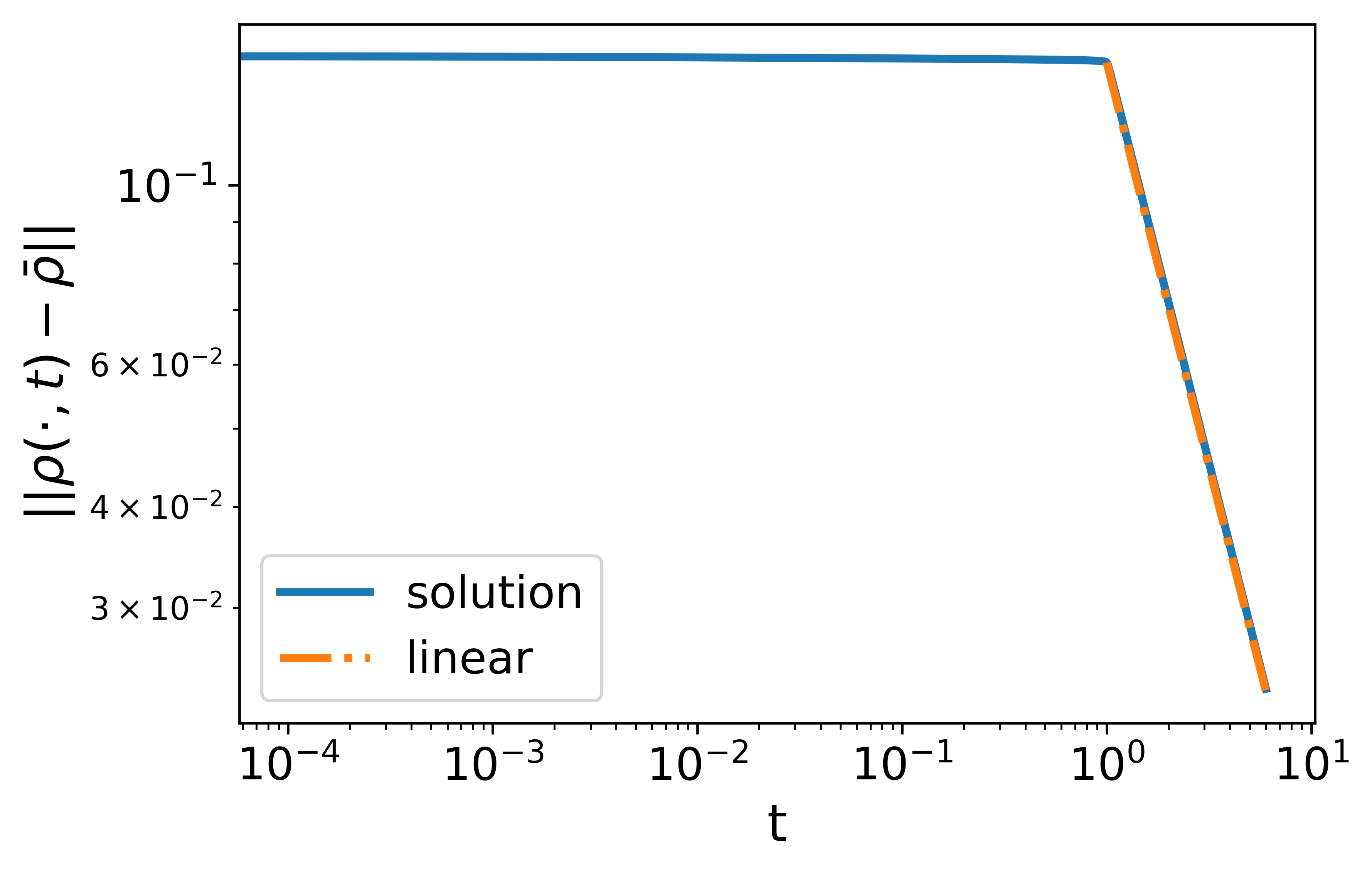}
  \end{subfigure}
  \begin{subfigure}{.32\textwidth}
    \includegraphics[width=\textwidth]{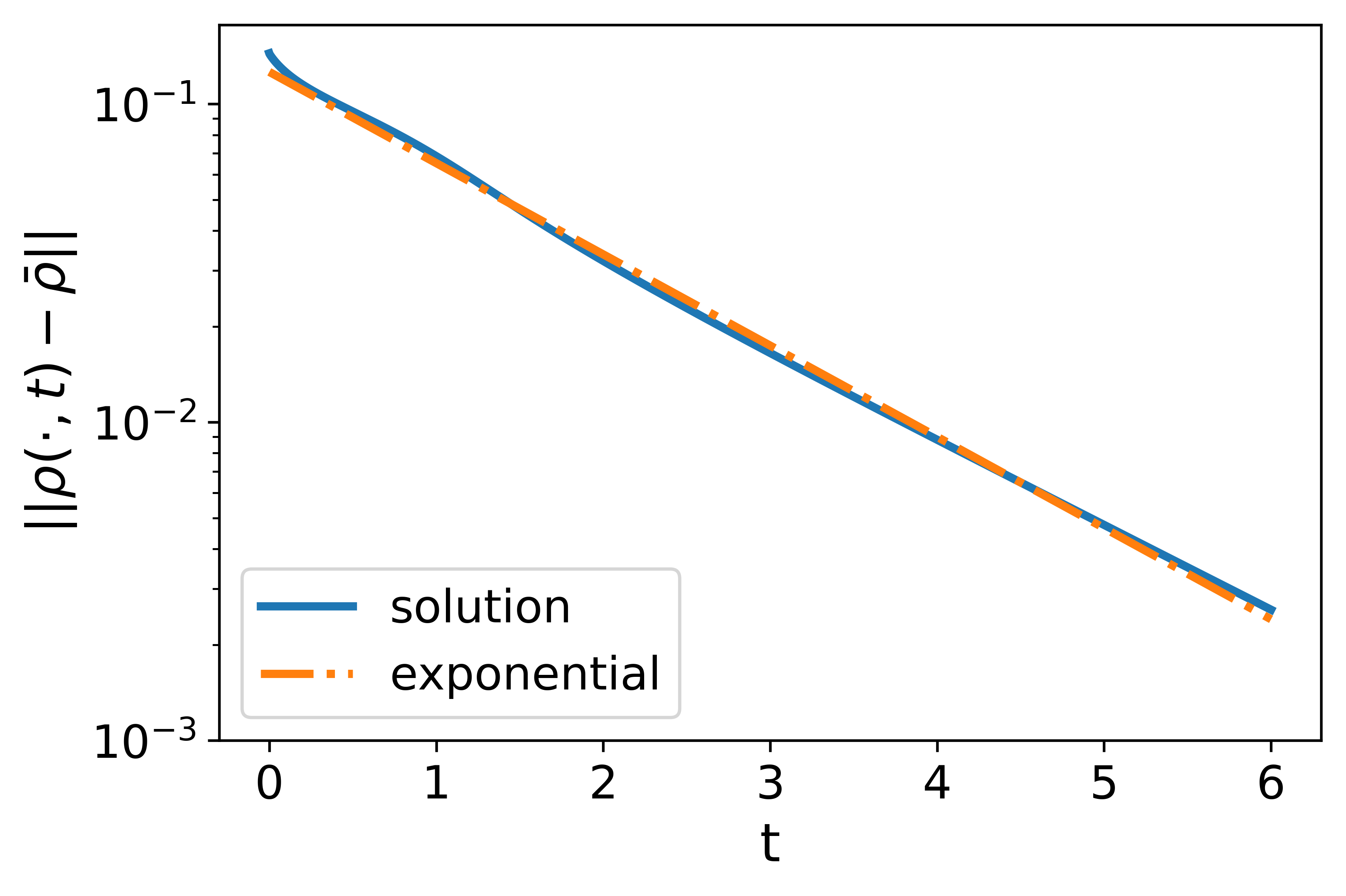}
  \end{subfigure}
  \begin{subfigure}{.32\textwidth}
    \includegraphics[width=\textwidth]{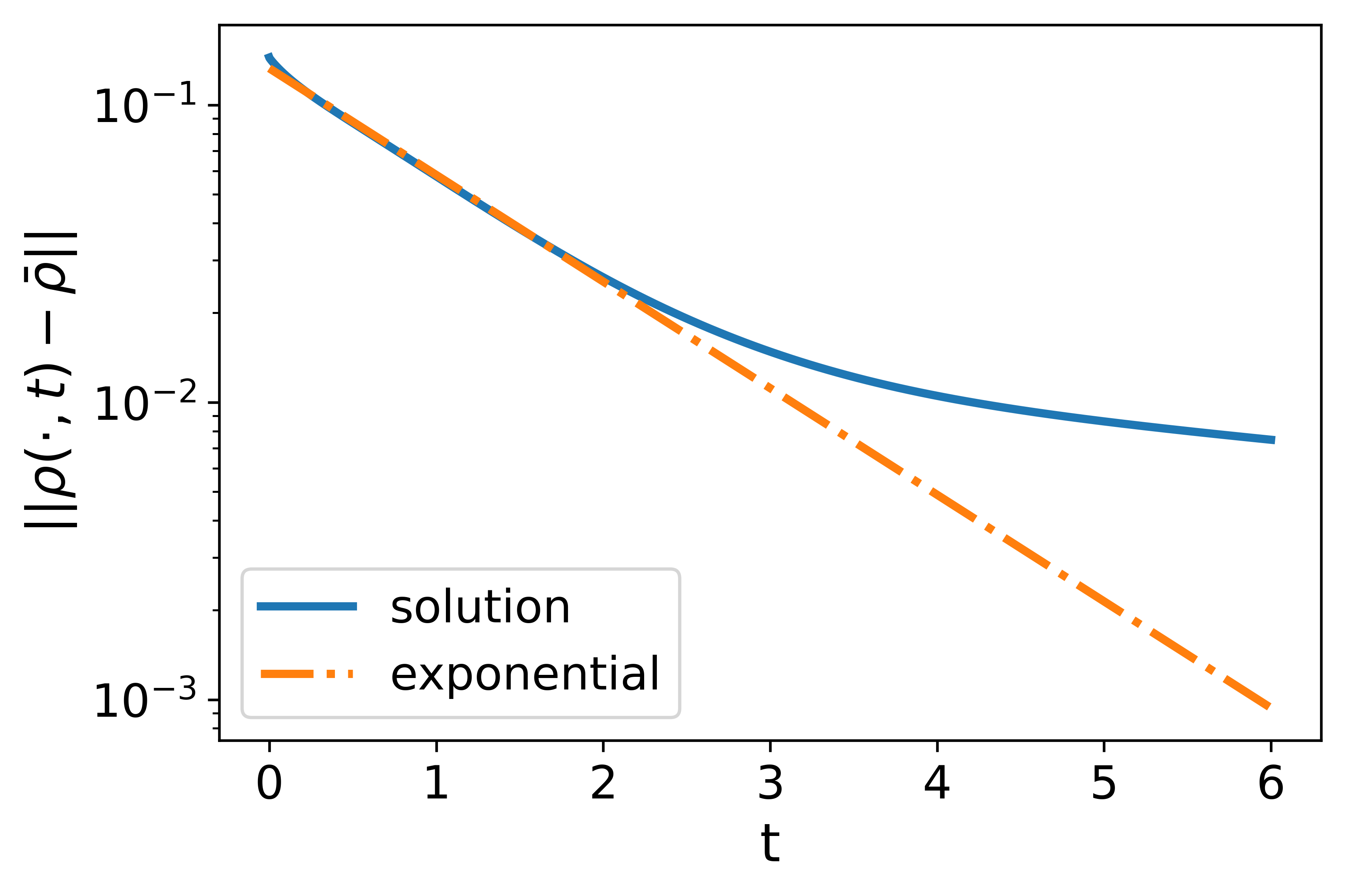}
  \end{subfigure}
  \caption{Compare solutions from different models with the linear initial data}
  \label{fig:linear}
\end{figure}

\begin{figure}[htbp]
  \centering
  \begin{subfigure}{.32\textwidth}
    \includegraphics[width=\textwidth]{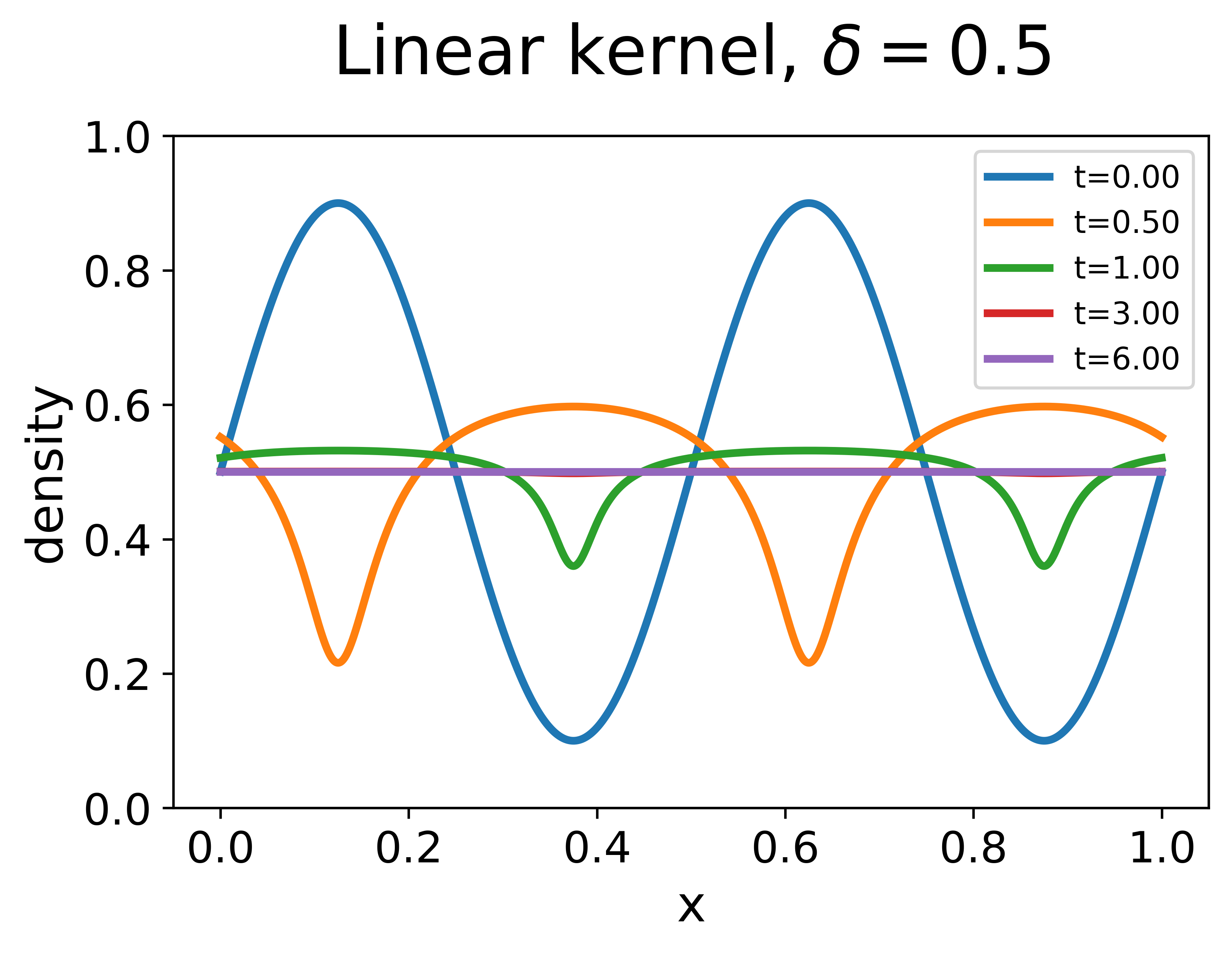}
  \end{subfigure}
  \begin{subfigure}{.32\textwidth}
    \includegraphics[width=\textwidth]{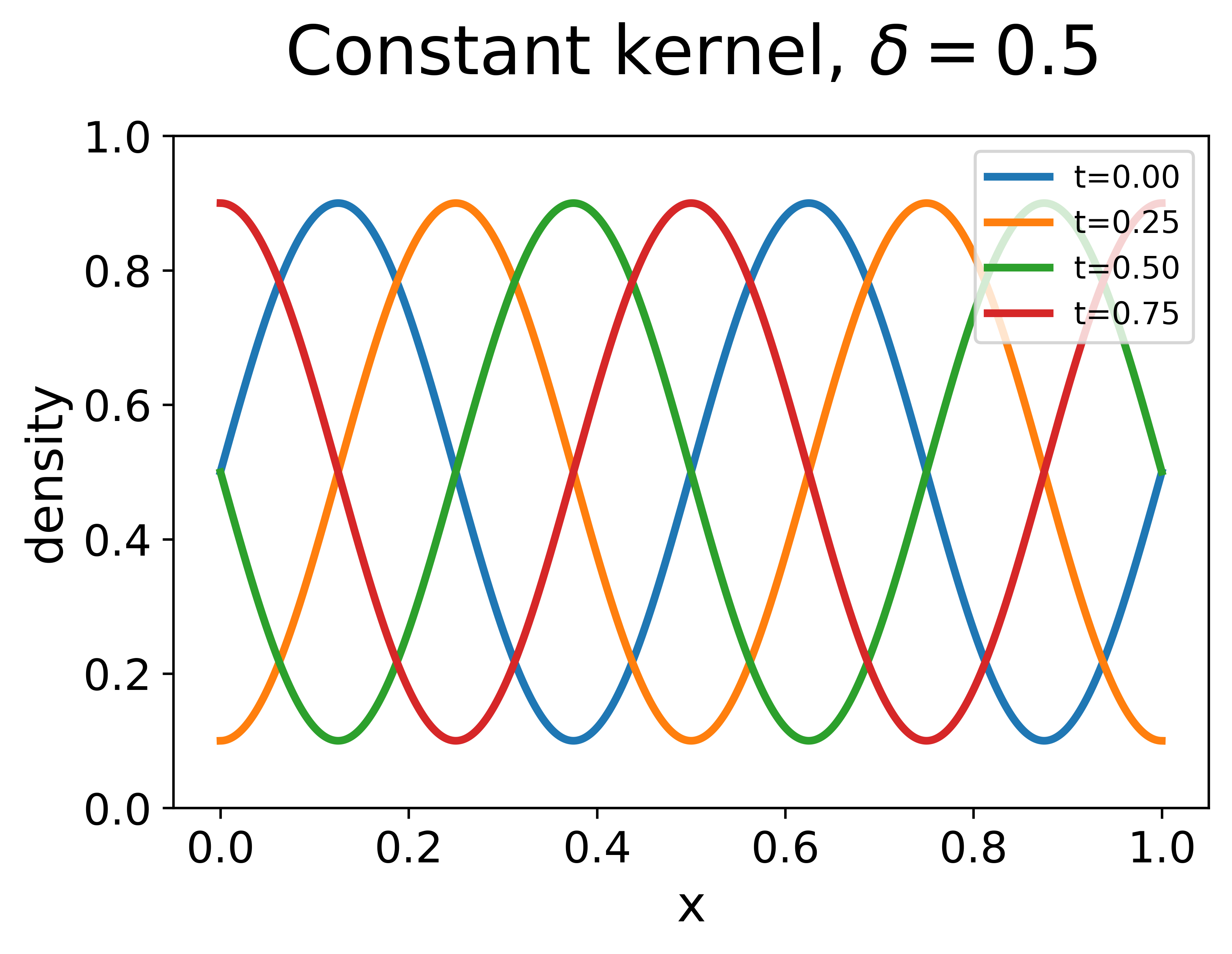}
  \end{subfigure}

  \begin{subfigure}{.32\textwidth}
    \includegraphics[width=\textwidth]{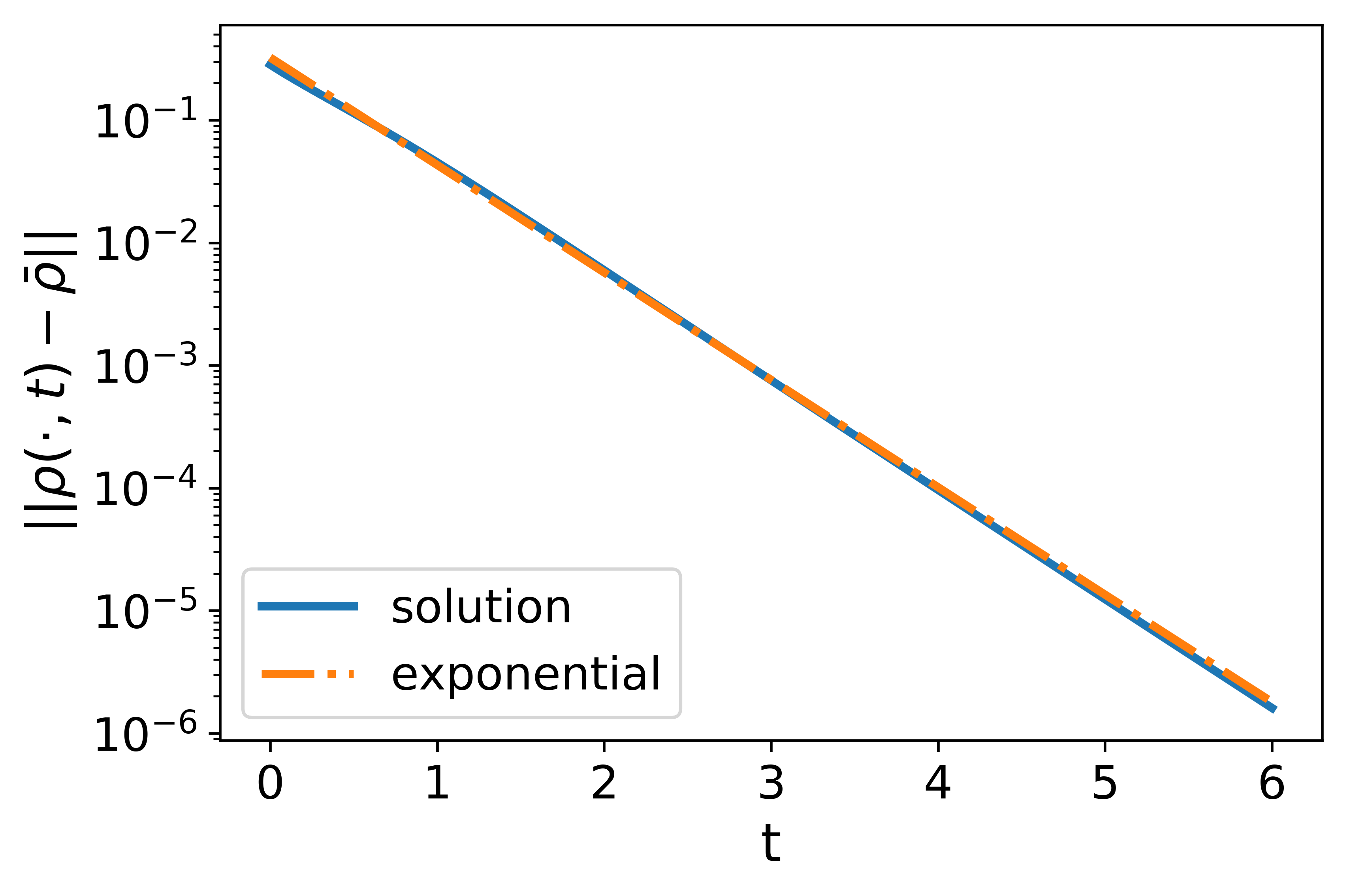}
  \end{subfigure}
  \begin{subfigure}{.32\textwidth}
    \includegraphics[width=\textwidth]{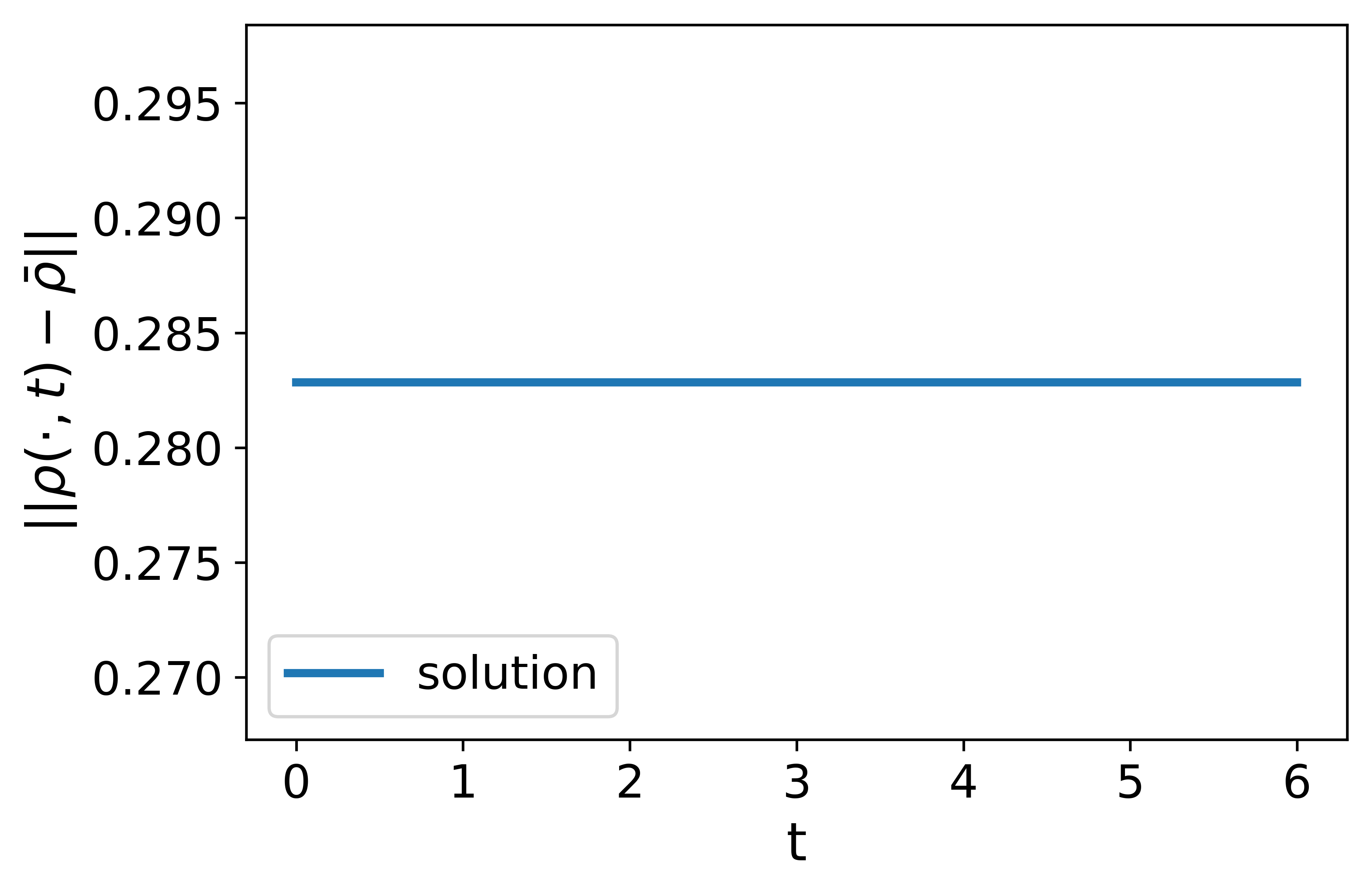}
  \end{subfigure}
  \caption{Compare solutions from different models with the sine-wave initial data}
  \label{fig:sinewave}
\end{figure}

\begin{figure}[htbp]
  \centering
  \begin{subfigure}{.32\textwidth}
    \includegraphics[width=\textwidth]{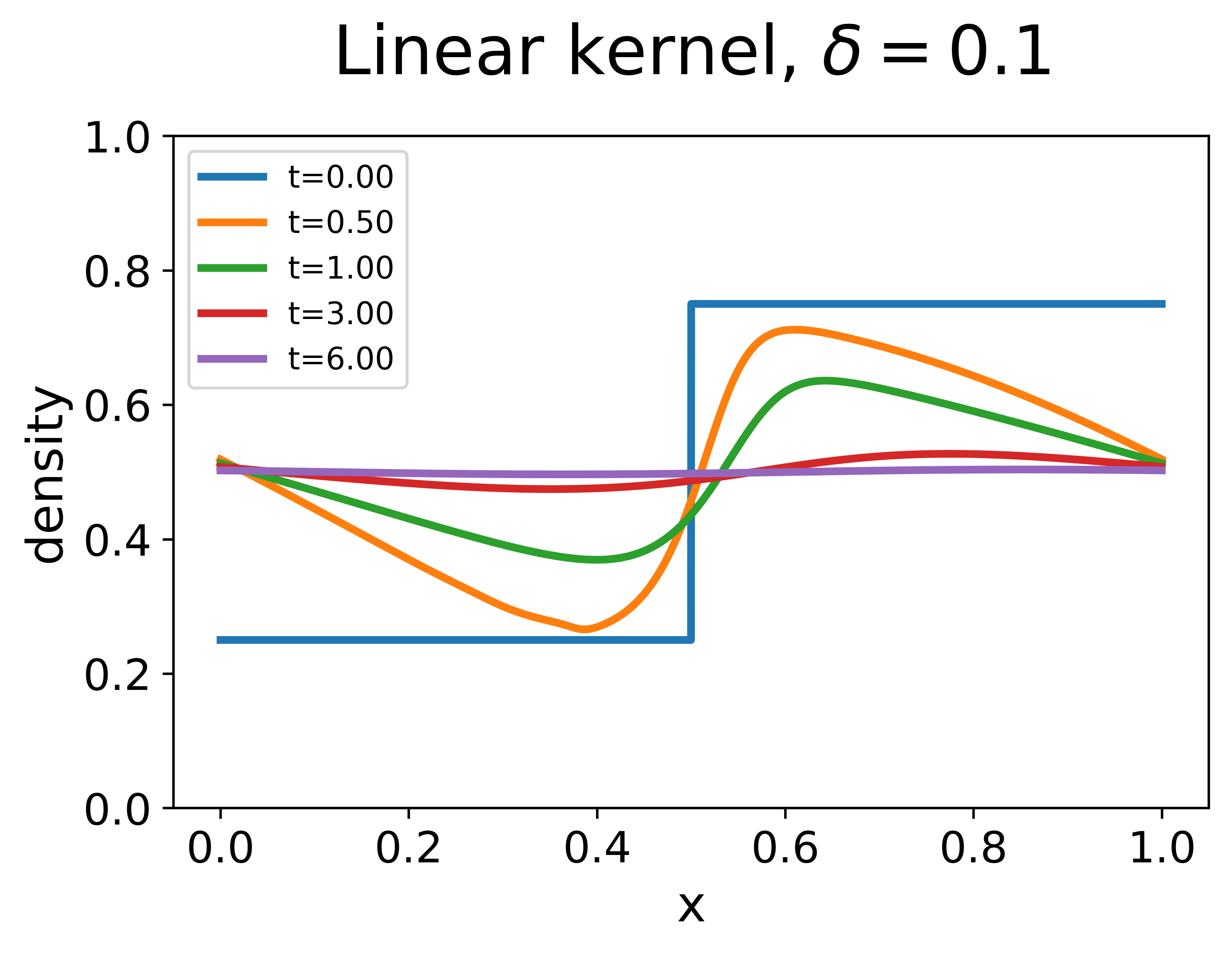}
  \end{subfigure}
  \begin{subfigure}{.32\textwidth}
    \includegraphics[width=\textwidth]{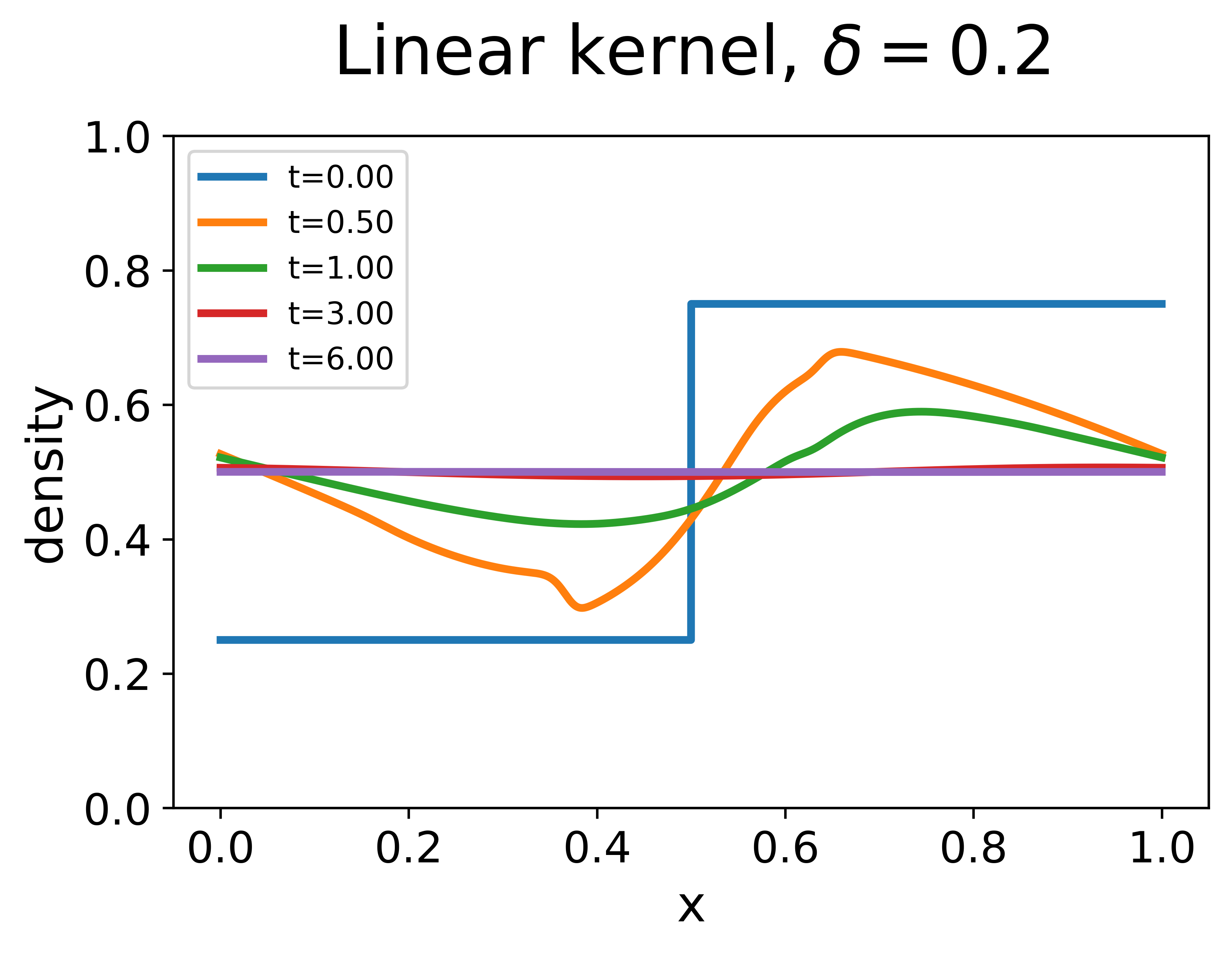}
  \end{subfigure}

  \begin{subfigure}{.32\textwidth}
    \includegraphics[width=\textwidth]{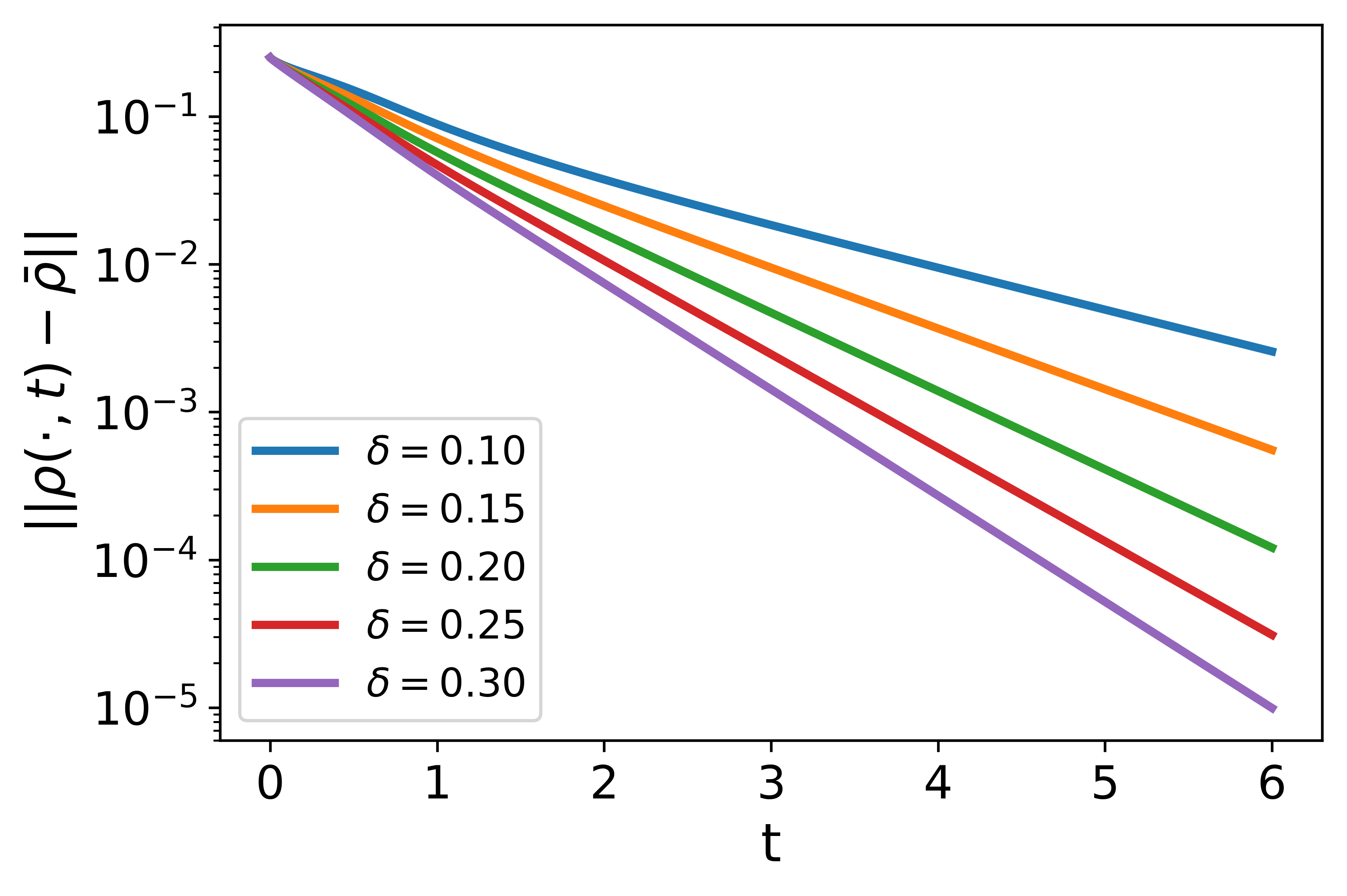}
  \end{subfigure}
  \begin{subfigure}{.32\textwidth}
    \includegraphics[width=\textwidth]{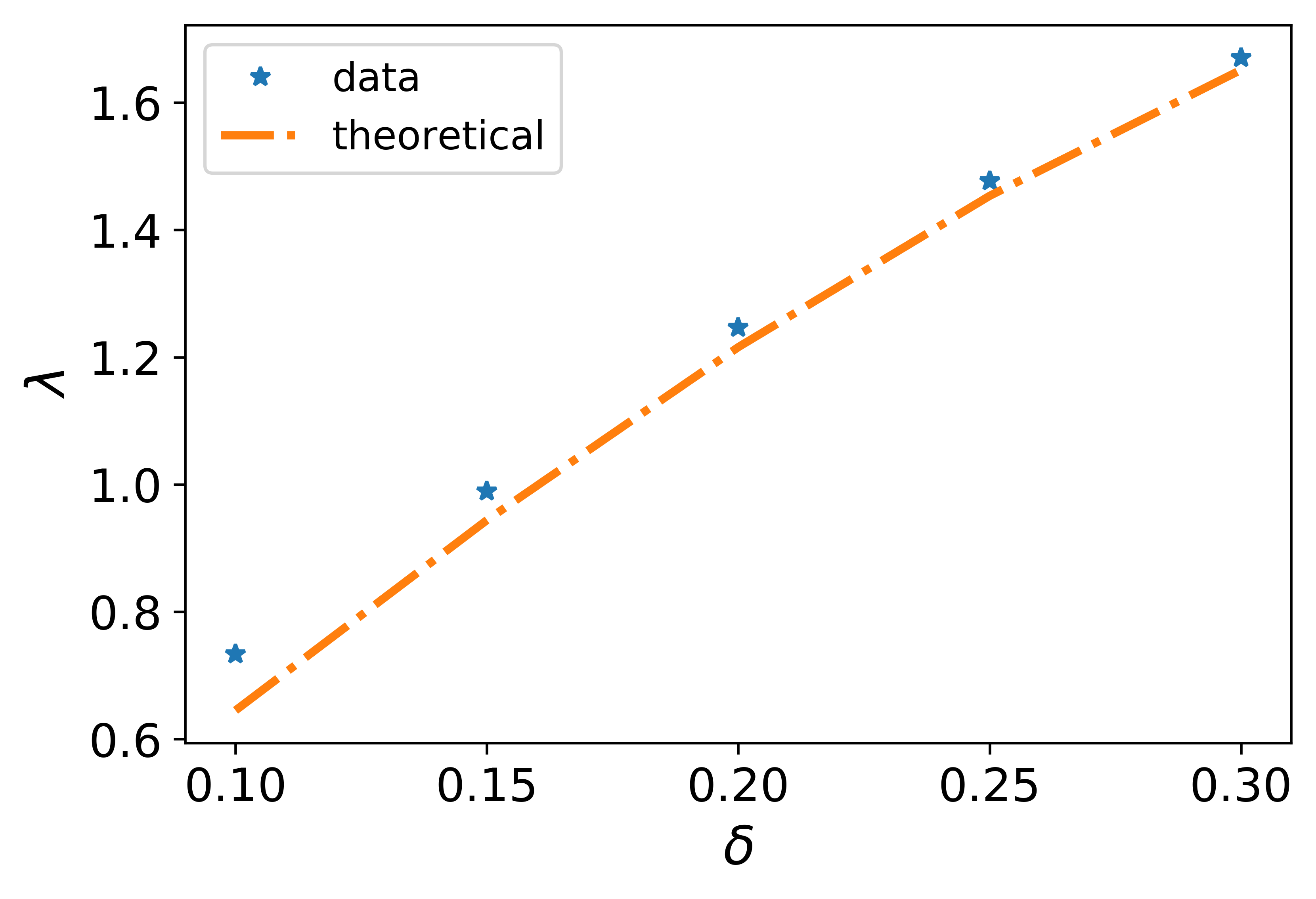}
  \end{subfigure}
  \caption{Compare solutions with different choices of nonlocal range $\delta$}
  \label{fig:piece_const}
\end{figure}

\begin{figure}[htbp]
	\centering
	\begin{subfigure}{.32\textwidth}
    	\includegraphics[width=\textwidth]{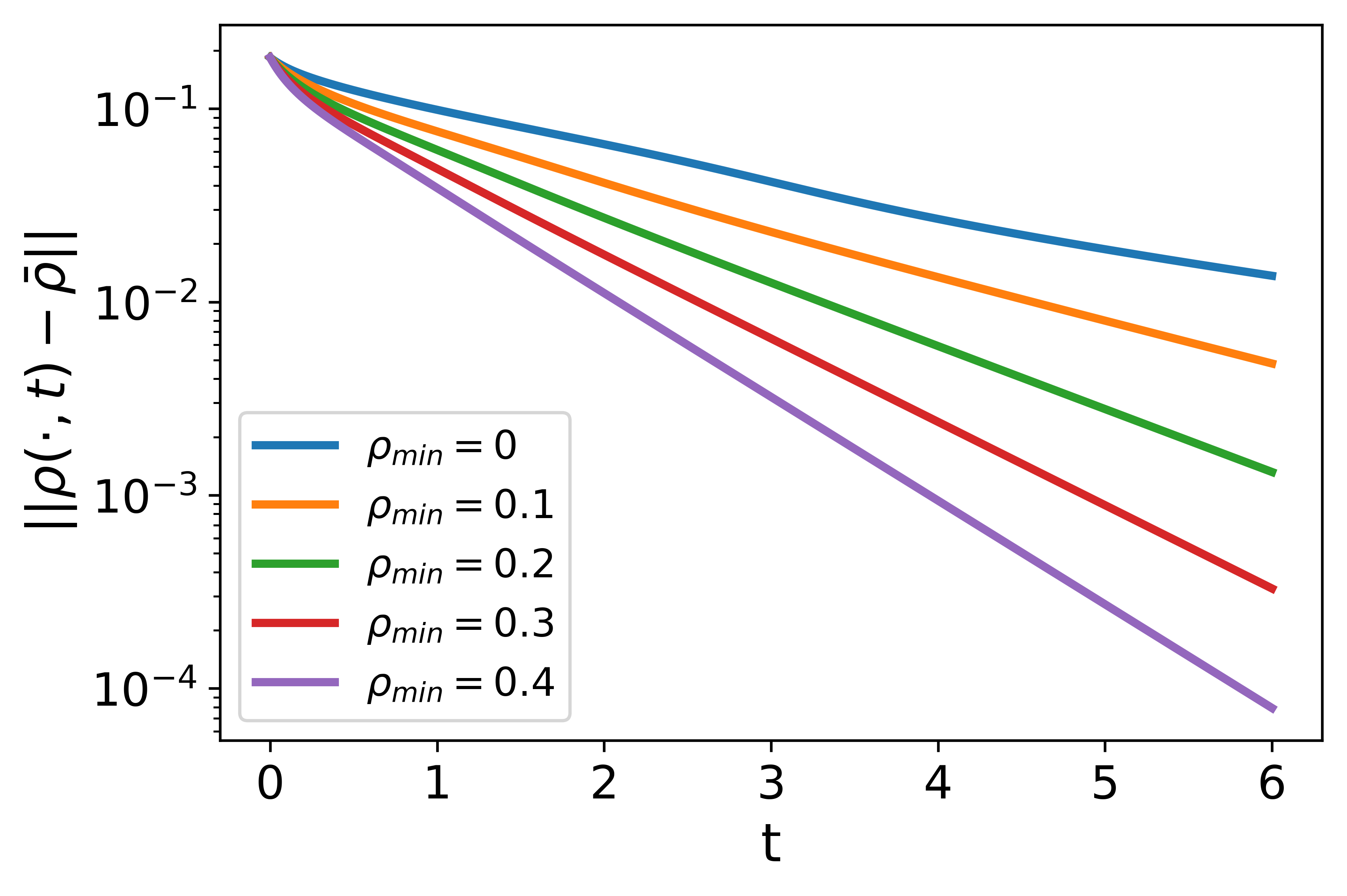}
  	\end{subfigure}
  	\begin{subfigure}{.32\textwidth}
    	\includegraphics[width=\textwidth]{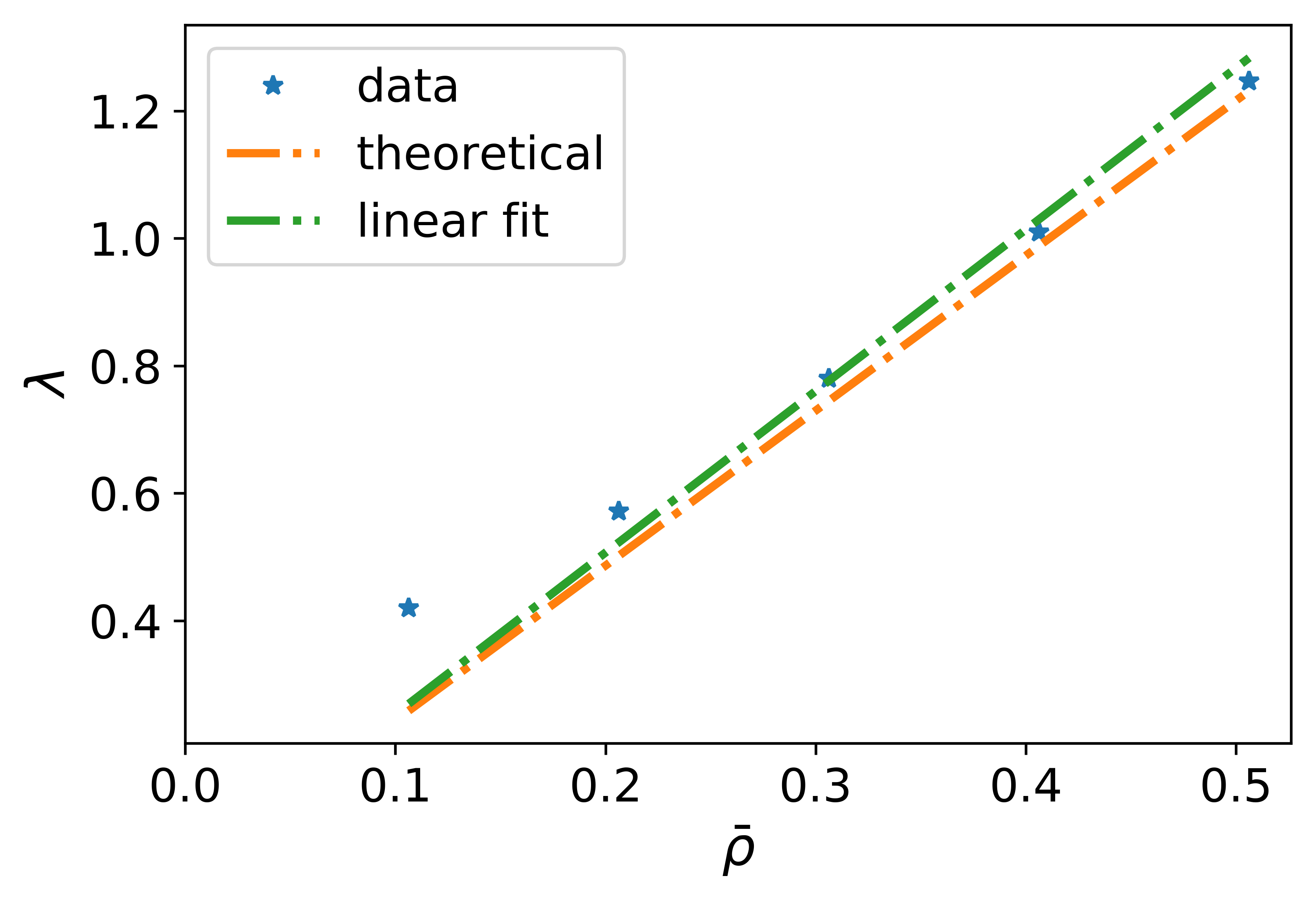}
  	\end{subfigure}
	\caption{Compare solutions with initial data with different mean densities}  
	\label{fig:ini}
\end{figure}

\section{Conclusions and future work}\label{sec:conclusion}

This paper studies global stability of a nonlocal traffic flow model, i.e., the nonlocal LWR that assumes vehicles' velocities depend on the nonlocal traffic density. Mathematically, the model is a scalar conservation law with a nonlocal term. Under some assumptions, we prove that the solution of the nonlocal LWR model converges exponentially to the uniform flow as time goes to infinity. The key assumption is that the nonlocal kernel should be non-increasing and non-constant. It reveals a simple but insightful principle for connected vehicle algorithm design that nearby information should deserve more attention. Indeed, equal attention (as associated with the constant kernel) might allow non-uniform traffic patterns to persist in time. Moreover, the analysis on the parameter dependence also shows the importance of 
choosing suitable ranges of nonlocal information for achieving the best effectiveness in traffic stabilization.

From the mathematical perspective, our proof relies on a couple of assumptions. We believe that the non-increasing and non-constant assumption on the nonlocal kernel plays the key role and other assumptions can be relaxed. For example, we have discussed how the regularity assumption on the solution can be relaxed. Further extensions can also be considered and the same global stability analysis may still be applicable for more general nonlinear desired speed functions as well as variants of other nonlocal macroscopic traffic models. It is also interesting to discuss asymptotic convergence to the uniform flow in other metrics. In particular, it will be interesting to consider
in the future the
$\mathbf{L}^1$  metric, which is popular for local conservation laws, and more general $\mathbf{L}^p$ metrics for  $p\neq 2$.

From the application perspective, it is an interesting question to check the generality of the proposed design principle for connected vehicles in real traffic.
For example, we are currently exploring the possible impact on utilizing   nonlocal information both in space and time.
Moreover, a realistic traffic system can be modeled on different scales using different models. For example, we may consider microscopic traffic models with a given number of discrete vehicles, or nonlocal traffic flow models based on Arrhenius type dynamics or having nudging (``look-behind'') terms. 
In addition, one may study how to extend the findings presented here to other traffic flow models involving both microscopic and macroscopic scales.

\section*{Acknowledgements} We thank the members of the CM3 group at Columbia University for fruitful discussions. 
We also thank the referees for very helpful suggestions.

\bibliographystyle{siam}
\bibliography{ref,nonlocal_conservation_law}
\end{document}